\documentclass[a4paper,11pt,twoside]{article}
\usepackage{amsmath,amssymb,amscd}
\usepackage{ascmac}
\usepackage[all]{xy}
\usepackage{enumerate}
\usepackage{theorem}

\newtheorem{thm}{Theorem}[section]
\newtheorem{prop}[thm]{Proposition}
\newtheorem{cor}[thm]{Corollary}
\newtheorem{lem}[thm]{Lemma}

{\theorembodyfont{\normalfont}
\newtheorem{rem}[thm]{Remark}

\newtheorem{dfn}[thm]{Definition}
\newtheorem{exa}[thm]{Example}

}

\makeatletter
\newenvironment{proof}[1][\proofname]{\par
  \normalfont
  \topsep6\p@\@plus6\p@ \trivlist
  \item[\hskip\labelsep{\textit{\mdseries #1}\@addpunct{\mdseries.}}]\ignorespaces
}{%
  \QED \endtrivlist
}
\newcommand{\proofname}{\normalfont{\textit{Proof.}}}
\makeatother
\makeatletter
\def\BOXSYMBOL{\RIfM@\bgroup\else$\bgroup\aftergroup$\fi
  \vcenter{\hrule\hbox{\vrule height.85em\kern.6em\vrule}\hrule}\egroup}
\makeatother
\newcommand{\BOX}{%
  \ifmmode\else\leavevmode\unskip\penalty9999\hbox{}\nobreak\hfill\fi
  \quad\hbox{\BOXSYMBOL}}
\newcommand\QED{\BOX}

\numberwithin{equation}{section}

\newtheorem{thm-dfn}[thm]{Theorem-Definition}

\newtheorem{thmalph}{Theorem}[section]
\newtheorem{coralph}[thmalph]{Corollary}

\def\Im{\mathop{\mathrm{Im}}\nolimits}
\def\Ker{\mathop{\mathrm{Ker}}\nolimits}

\def\Hom{\mathop{\mathrm{Hom}}\nolimits}

\def\End{\mathop{\mathrm{End}}\nolimits}

\def\dom{\mathop{\mathrm{dom}}\nolimits}

\def\supp{\mathop{\mathrm{supp}}\nolimits}
\def\ind{\mathop{\mathrm{ind}}\nolimits}

\def\id{\mathord{\mathrm{id}}}

\def\calq{\mathop{\mathcal{Q}}\limits}

\newcommand{\bb}[1]{\mathbb{#1}}

\newcommand{\C}{\mathbb{C}}
\newcommand{\R}{\mathbb{R}}
\newcommand{\Z}{\mathbb{Z}}
\newcommand{\N}{\mathbb{N}}

\newcommand{\ep}{\varepsilon}
\newcommand{\dif}{\mathrm{d}}

\newcommand{\relmiddle}[1]{\mathrel{}\middle#1\mathrel{}}
\newcommand{\map}[3]{ #1 \colon #2 \to #3 }
\newcommand{\brackets}[1]{ \left[ #1 \right]}

\newcommand{\braces}[1]{ \left\{ #1 \right\}}
\newcommand{\parens}[1]{ \left( #1 \right)}
\newcommand{\angles}[1]{ \left \langle #1 \right \rangle}
\newcommand{\norm}[1]{ \left \Vert #1 \right \Vert}
\newcommand{\set}[2]{ \left \{ #1 \relmiddle| #2 \right \}}
\newcommand{\half}{\frac{1}{2}}

\newcommand{\E}{\mathcal{E}}
\newcommand{\Max}{\mathrm{Max}}
\newcommand{\red}{\mathrm{red}}
\newcommand{\qqand}{\qquad \text{and}\qquad}
\newcommand{\qqfor}{\qquad \text{for}\quad}
\newcommand{\hattensor}{\widehat{\otimes}}
\newcommand{\turnV}[1]{\bigwedge\nolimits^{\! #1}}

\newcommand{\gm}{\gamma}
\newcommand{\gminv}{\gamma_{}^{-1}}
\newcommand{\EG}{\mathcal{E}G}

\newcommand{\settingsGXa}
{Let $G$ be a second countable locally compact Hausdorff group.
Let $X$ be a $G$-compact 
proper complete $G$-Riemannian manifold. }

\newcommand{\dirsum}[1]{\bigoplus #1}

\newcommand{\bmatrixtwo}[4]{
\begin{bmatrix} #1 & #2 \\ #3 & #4 \end{bmatrix} }
\newcommand{\smatrixtwo}[4]{
\begin{smallmatrix} #1 & #2 \\ #3 & #4 \end{smallmatrix} }

\title{$G$-Homotopy Invariance of the Analytic Signature of Proper Co-compact $G$-manifolds and Equivariant Novikov Conjecture}
\author{Yoshiyasu Fukumoto \footnote
{
Research Center for Operator Algebras, Department of Mathematics, 
East China Normal University
\newline \, \, \,
3663 North ZhongShan Road, Shanghai, CHINA. 200062
\newline \, \, \,
fukumoto@math.ecnu.edu.cn}
}
\date{Research Center for Operator Algebras, Department of Mathematics,
\\
East China Normal University}
\begin{document}
\maketitle
\vspace{-1ex}
\begin{abstract}
The main result of this paper is the $G$-homotopy invariance of the $G$-index of 
signature operator of proper co-compact $G$-manifolds.
If proper co-compact $G$ manifolds $X$ and $Y$ are $G$-homotopy equivalent,
then we prove that the images of their signature operators by the $G$-index map are the same
in the $K$-theory of the $C^{*}$-algebra of the group $G$.
Neither discreteness of the locally compact group $G$ nor freeness of the action 
of $G$ on $X$ are required,
so this is a generalization of the classical case of closed manifolds.
Using this result we can deduce the equivariant version of Novikov conjecture for proper co-compact 
$G$-manifolds from the Strong Novikov conjecture for $G$. 
\end{abstract}
\tableofcontents
\section*{Introduction}
Before discussing on our case of proper $G$-action, let us review the 
classical case of closed manifolds.
For even dimensional oriented closed manifold $M$, 
The ordinary Fredholm index of 
the signature operator $\partial_{M}^{}$ is equal to the signature of the manifold $M$
which is defined using the cup product of the ordinary cohomology of $M$.
In particular it follows that $\mathrm{ind}(\partial_{M}^{})$ is invariant
under orientation preserving homotopy.
We have the following classical and important result;
\begin{thm}{\upshape \cite{Kas75}, \cite{KamMil85}, \cite{HilSk92}}
\label{ThmClassical}
Let $M$ and $N$ be even dimensional oriented closed manifolds with fundamental group
$\Gamma = \pi_{1}^{}(M) = \pi_{1}^{}(N)$.
Assume that $M$ and $N$ are orientation preserving homotopy equivalent to each other.
Then $\mathrm{ind}_{\Gamma}^{}(\partial_{M}^{}) = \mathrm{ind}_{\Gamma}^{}(\partial_{N}^{}) \in K_{0}^{}(\Gamma)$.
\end{thm}
Notice that we can deduce the Novikov conjecture from the Strong Novikov conjecture
by using this theorem.
Moreover, we also have a more generalized result;
\begin{thm}{\upshape \cite[3.3. PROPOSITION and 3.6. THEOREM]{RoWe90}}
Let a finite group $\Gamma$ acts on $M$ and $N$ and 
let $\Gamma = \pi_{1}^{}(M) = \pi_{1}^{}(N)$.
Let $\ind_{\Gamma}^{G} $ be the $G$-equivariant $\Gamma$-index map
with value in $K_{0}^{G}(C_{\red}^{*}(\Gamma)) \simeq K_{0}^{}(C_{\red}^{*}(G_{}^{\Gamma}))$,
where $G_{}^{\Gamma}$ denotes the group extension 
$\{ 1 \} \to \Gamma \to G_{}^{\Gamma} \to G \to \{ 1 \}$.
Assume that $M$ and $N$ are orientation preserving $\Gamma$-equivariantly homotopy equivalent.
Then $\ind_{\Gamma}^{G} (\partial_{M}^{}) 
= \ind_{\Gamma}^{G} (\partial_{N}^{}) \in K_{0}^{}(C_{\red}^{*}(G_{}^{\Gamma}))$.
\end{thm}

Our main theorem is a generalization of them. 
Let us fix the settings.
Let $X$ and $Y$ be oriented even-dimensional complete Riemnannian manifolds
and let $G$ be a second countable locally compact Hausdorff group
acting on $X$ and $Y$ isometrically, properly and co-compactly.
\begin{thmalph}
\label{MainThm}
Let $X$ and $Y$ be oriented even-dimensional complete Riemnannian manifolds
and let $G$ be a second countable locally compact Hausdorff group
acting on $X$ and $Y$ isometrically, properly and co-compactly.
Let $\partial_{X}^{}$ and $\partial_{Y}^{}$ be the signature operators.
Assume the we have 
a $G$-equivariant orientation preserving homotopy equivalent map $\map{f} {Y} {X}$.

Then $\mathrm{ind}_{G}^{} (\partial_{X}^{}) = \mathrm{ind}_{G}^{} (\partial_{Y}^{} ) \in K_{0}^{}(C^{*}(G))$.
\end{thmalph}
This claim is also stated in \cite{Ba-Co-Hi76} without proofs
and here we will give a proof for it to obtain Corollary \ref{MainCor}.
The method we use in this paper is based on \cite{HilSk92},
so we will construct a map that sends $\mathrm{ind}_{G}^{} (\partial_{X}^{})$
to $\mathrm{ind}_{G}^{} (\partial_{Y}^{} )$.
Our group $C^{*}$-algebras can be either maximal one or reduced one.

Theorem \ref{ThmClassical} is the case when $X$ and $Y$ are the universal covering of 
closed manifolds $M$ and $N$.
Thus, analogously to the case of closed manifolds, 
the equivariant version of the Novikov conjecture 
can be deduced from the Strong Novikov conjecture for the acting group $G$.
In particular, by using this theorem and the result discussed in \cite{Fu16},
we obtain the following equivariant version of Novikov conjecture for low dimensional cohomologies:
\begin{coralph}
\label{MainCor}
Let $X$, $Y$ and $G$ as above and let $L$ be a $G$-hermitian line bundles over $X$
which is induced from a $G$-line bundle over $\EG$, or more generally, 
$G$-hermitian line bundle $L$ over $X$ satisfying $c_{1}(L) = 0 \in H_{}^{2}(X; \R)$.
Suppose, in addition, that $G$ is unimodular and $H_{1}^{}(X; \R) = H_{1}^{}(Y; \R) = \{ 0 \}$.
Then,
$$
\int_{X} c_{X}^{}(x) \mathcal{L}(TX)\wedge \mathrm{ch}(L)
=
\int_{Y} c_{Y}^{}(y) \mathcal{L}(TY)\wedge \mathrm{ch}(f_{}^{*}L),
$$
\end{coralph}
where $c_{X}^{}$ denotes the cut-off function,
that is,  $c_{X}^{}$ is a $\R_{\geq 0}^{}$-valued
compactly supported function on $X$ satisfying
$\int_{G}c(\gminv x) \: \dif \gm = 1$ for any $x \in X$.
In the case of the closed manifold, that is, when $X$ is obtained as
the universal covering of a closed manifold $M$,
and the acting group is the fundamental group,
the above value is equal to the ordinary, so called, higher signature
$\angles{ \mathcal{L}(TX) \cup \mathrm{ch}(L),\ [M] }$.
The same result in this case of closed manifolds
was obtained in \cite{Ma03} and \cite{HaSc08}.

Moreover in Section \label{SectionAlmFlat},
we will prove the $G$-homotopy invariance of the analytic signature
twisted by almost flat bundles as in \cite[Section 4.]{HilSk92}.
However we will use a different method from \cite{HilSk92}
to deal with general $G$-invariant elliptic operators.
To be specific, we will prove the following Theorem \ref{ThmAlmostFlatIndex} 
to obtain Corollary \ref{CorAlmostFlatSgn}.
\begin{thmalph}
\label{ThmAlmostFlatIndex}
Let $X$ be a complete oriented Riemannian manifold
and let $G$ be a locally compact Hausdorff group
acting on $X$ isometrically, properly and co-compactly.
Moreover we assume that $X$ is simply connected.
Let $D$ be a $G$-invariant properly supported elliptic operator of order $0$
on $G$-Hermitian vector bundle over $X$.

Then there exists $\ep > 0$ satisfying the following:
for any finitely generated projective
Hilbert $B$-module $G$-bundle $E$ over $X$
equipped with a $G$-invariant Hermitian connection
such that $\norm{R_{}^{E}} < \ep$,
we have 
$$
\mathrm{ind}_{G}^{}\parens{ [E] \hattensor_{C_{0}^{}(X)}^{} [D] }= 0 
\quad \in K_{0}^{} \parens{C_{\Max}^{*}(G) {\otimes_{\Max}^{}} B}
$$
if $\mathrm{ind}_{G}^{}([D]) =0 \in K_{0}^{}(C_{\Max}^{*}(G))$.
If we only consider commutative $C^{*}$-algebras for $B$, then the same conclusion
is also valid for $C_{\red}^{*}(G)$.
\end{thmalph}
\begin{coralph}
\label{CorAlmostFlatSgn}
Consider the same conditions as Theorem \ref{MainThm} on $X$, $Y$ and $G$
and assume additionally that $X$ and $Y$ are simply connected.

Then there exists $\ep > 0$ satisfying the following:
for any finitely generated projective
Hilbert $B$-module $G$-bundle $E$ over $X$ equipped with a $G$-invariant Hermitian connection
such that $\norm{R_{}^{E}} < \ep$,
we have 
$$
\mathrm{ind}_{G}^{} ([E]\hattensor [\partial_{X}^{}])
 = \mathrm{ind}_{G}^{} ([f_{}^{*}E]\hattensor [\partial_{Y}^{}]) 
\quad \in K_{0}^{}(C_{\Max}^{*}(G)\hattensor_{\Max}^{} B).
$$
If we only consider commutative $C^{*}$-algebras for $B$, then the same conclusion
is also valid for $C_{\red}^{*}(G)$.
\end{coralph}
\section{Preliminaries on proper actions}
\begin{dfn}
Let $G$ be a second countable locally compact Hausdorff
group.
Let $X$ be a complete Riemannian manifold.
\begin{itemize}
\item
$X$ is called a $G$-Riemannian manifold
if $G$ acts on $X$ isometrically.
\item
The action of $G$ on $X$ is said to be proper
or $X$ is called a proper $G$-space
if the following continuous map is proper:
$
X\times G \to X \times X, \quad
(x, \gm)\mapsto (x, \gm x).
$
\item
The action of $G$ on $X$ is said to be co-compact
or
$X$ is called $G$-compact space if the quotient space
$X/G$ is compact.
\end{itemize}
\end{dfn}
\begin{dfn}
The action of $G$ on $X$
induces actions on $TX$ and $T_{}^{*}\!X$ given by
$$
\begin{array}{cccc}
&\gm \colon T_{x}^{}X &\to &T_{\gm x}^{}X \qquad \quad
\\
&\qquad v &\mapsto &\gm(v):=\gm_{*}^{}v
\end{array}
\qqand
\begin{array}{ccccc}
&\gm \colon T^{*}_{x}X &\to &T^{*}_{\gm x}X \qquad \qquad \quad
\\
&\qquad \xi &\mapsto &\gm(\xi):=(\gminv)_{}^{*}\xi.
\end{array}
$$
The action on $\mathfrak{X}(X)$ and $\Omega_{}^{*}(X)$ is given by
$$
\gm[V]
:= 
\gm_{*}^{}V
\qqand
\gm[\omega]
:= 
({\gminv})_{}^{*} \omega
$$
for $\gm\in G$, $V\in\mathfrak{X}(X)$
and $\omega \in \Omega_{}^{*}(X)$. Obviously,
$\gm[\omega \wedge \eta] = \gm[\omega] \wedge \gm[\eta]$
and $\dif (\gm [\omega]) = \gm[\dif \omega]$.
\end{dfn}
\begin{prop}[Slice theorem]\label{SliceThm}
Let $G$ be a second countable locally compact Hausdorff
group and act
properly and isometrically on $X$.
Then for any neighborhood $O$ of any point $x\in X$
there exists a compact subgroup $K \subset G$
including the stabilizer at $x$,
$K \supset G_{x} := \set{\gm \in G}{\gm x =x}$
and
there exists a $K$-slice $\{ x \} \subset S \subset O$.
\QED
\end{prop}
Here $S \subset X$ is called $K$-slice if the followings 
are satisfied;
\begin{itemize}
\item
$S$ is $K$-invariant; $K(S) = S$,
\item
the tubular subset $G(S) \subset X$ is open,
\item
there exists a $G$-equivariant map
$\map{\psi}{G(S)}{G/K}$ satisfying $\psi^{-1}([e]) = S$,
called a slice map.
\end{itemize}
\begin{cor}\label{CorSliceThm}
We additionally assume that $X$ is $G$-compact.
Then for any open covering
$X = \bigcup_{x \in X} O_{x}$,
there exists a sub-family of
finitely many open subsets $\{O_{x_i}, \ldots O_{x_N}\}$
such that
$$\bigcup_{\gm \in G} \bigcup_{i=1}^{N} \gm(O_{x_i})=X.$$

In particular, $X$ is of bounded geometry,
namely, the injective radius is bounded below and
the norm of Riemannian curvature is bounded.
\QED
\end{cor}
\begin{lem}
\label{LemPropernessOfEquivMap}
Let $X$ and $Y$ be manifolds on which $G$ acts properly.
Suppose that the action on $Y$ is co-compact.
Let $\map{f}{Y}{X}$ be a $G$-equivariant continuous map.
Then $f$ is a proper map. 
\end{lem}
\begin{proof}
Since the action on $Y$ is co-compact, there exists a compact subset $F \subset Y$
satisfying $G(F) = Y$.
Fix a compact subset $C \subset X$ and assume that
the closed set $f_{}^{-1}C \subset Y$ is not compact.
Then there exists a sequence $\{ y_{j}^{} \} \subset f_{}^{-1}C$ tending to the infinity,
that is, any compact subset in $Y$ contains only finitely many points of $\{y_{j}^{} \}$.
Since the action on $Y$ is proper, there exists a sequence $\{ \gm_{j}^{} \} \subset G$
tending to the infinity satisfying $y_{j}^{} \in \gm_{j}^{} F$.
Then it follows that $f(y_{j}^{}) \in f(\gm_{j}^{} F) = \gm_{j}^{} f(F)$.
Due to the compactness of  $f(F) \subset X$ and the properness of the action on $X$,
the sequence $\{ f(y_{j}^{}) \} \subset X$ tends to the infinity.
However, the compact subset $C$ cannot contain such a sequence.
So, $f_{}^{-1}C$ is compact.
\end{proof}
\section{Perturbation arguments}
In this section, we will discuss on some technical method
introduced in \cite[Section 1 and 2]{HilSk92}.
For now, we will forget about the manifolds and group actions.
Let $A$ be a $C^*$-algebra, which may not be unital.
Especially we will consider $A = C^{*}(G)$.
Let $\E$ be a Hilbert $A$-module equipped with $A$-valued scalar product
$\angles{\cdot, \cdot}$. 
Let us fix some notations:
\begin{itemize}
\item
$\mathbb{L}(\E_{1}^{}, \E_{2}^{})$ denotes a space consisting of adjointable $A$-linear operators,
and we also use $\mathbb{L}(\E) := \mathbb{L}(\E, \E)$.
\item
$\mathbb{K}(\E_{1}^{}, \E_{2}^{})$ denotes a sub-space of  $\mathbb{L}(\E_{1}^{}, \E_{2}^{})$
consisting of compact $A$-linear operators,
namely, the norm closure of the space of operators whose $A$-rank are finite. 
We also use $\mathbb{K}(\E) := \mathbb{K}(\E, \E)$.
\end{itemize}

\subsection{Quadratic forms and graded modules}
\begin{dfn}[Regular quadratic forms]
$\map{Q}{\E \times \E}{A}$ is called a quadratic form on $\E$ if it satisfies 
\begin{align}
Q(\xi, \nu) = Q(\nu, \xi)^{*} \qqand Q(\nu, \xi a) = Q(\nu, \xi)a
\qqfor \nu, \xi \in \E,\ a \in A. 
\end{align}
A quadratic form $Q$ is said to be regular if there exists a invertible operator
$B \in \mathbb{L}(\E)$ satisfying that
$Q(\xi, B \nu) = \angles{\xi, \nu} $.

For an operator $T \in \mathbb{L}(\E)$, let $T^{\prime}$ denote the adjoint with respect to
$Q$, that is, an operator satisfying that $Q(T \xi, \nu) = Q(\xi, T^{\prime} \nu)$.
Using $B$, it is written as $T^{\prime} = BT^{*}B^{-1}$.
\end{dfn}

\begin{dfn}[Compatible scalar product]
Another scalar product $\angles{\cdot, \cdot}_{1}^{} \colon \E \times \E \to A$ is called 
compatible with $\angles{\cdot, \cdot}$ 
if there exists a linear bijection $P \colon \E \to \E$
satisfying that $\angles{\nu, \xi}_{1}^{} = \angles{\nu, P \xi }$.
\end{dfn}
Note that
$P$ is a positive operator with respect to both of the scalar product
and 
$\sqrt{P} \colon (\E, \angles{\cdot, \cdot}_{1}^{}) \to (\E, \angles{\cdot, \cdot})$ 
is a unitary isomorphism.
In particular, neither the spaces $\mathbb{L}(\E)$ nor $\mathbb{K}(\E)$ depends on
the choice of compatible scalar product.

\begin{lem}
\label{LemAssociatedScalarProd}
Let $Q$ be a regular quadratic form on $\E$.
then there exist a compatible scalar product $\angles{\cdot, \cdot}_{Q}^{}$ with the initial
scalar product of $\E$
and $U \in \mathbb{L}(\E)$ satisfying that
$Q(\xi, U \nu) = \angles{\xi, \nu}_{Q}^{}$ and $U^{2} = 1$.
Moreover they are unique. 
\end{lem} 
\begin{proof}
With respect to the initial scalar product $\angles{\cdot, \cdot}$, we have that 
\begin{align*}
\angles{\nu,\ B^{-1}\xi} = Q(\nu,\ \xi) = Q(\xi,\ \nu)^{*}
=\angles{\xi,\ B^{-1} \nu}^{*} = \angles{B^{-1} \nu,\ \xi} = \angles{\nu,\ (B^{-1})^{*} \xi},
\end{align*}
which implies that $B^{-1}$ is an invertible self-adjoint operator.
Thus, it has the polar decomposition $B^{-1}=UP$
in which $B^{-1}$, $U$ and $P$ commute one another, here $U$ is unitary and $P$ is positive.
To be specific, $U$ and $P$ are given by the continuous functional calculus.
Let $f$ and $g$ be continuous functions given by
$f(x) := \frac{x}{|x|}$ and $g(x) := |x|$
on the spectrum of $B^{-1}$, which is 
contained in $\R \setminus \{ 0\}$, and set $U:=f(B^{-1})$ and $P := g(B^{-1})$.
Note that 
$$U= P^{-1} B^{-1} = P^{-1} (B^{-1})^{*} = P^{-1}PU^{*} = U^{*},$$
so it follows that $U^{2} = U^{*}U = 1$.
Let us set
$\angles{\nu, \xi}_{Q}^{} := \angles{\nu, P\xi}$.
Then, 
\begin{align*}
Q(\nu,\ U\xi) = Q(\nu,\ U^{-1}\xi)  = Q(\nu,\ B P\xi)
 = \angles{ \nu,\ P \xi} = \angles{ \nu,\ \xi}_{Q}^{}.
\end{align*}
If there is another such operator $U_{1}^{}$ satisfying that $U_{1}^{2}=1$ and that
$Q(\nu, U_{1}^{} \xi)$ is another scalar product,
then $U_{1}^{-1}U$ is a positive unitary operator, which implies that $U_{1}^{-1}U = 1$.
Thus we obtained the uniqueness.
\end{proof}
\begin{rem}
A regular quadratic form $Q$ on a Hilbert $A$-module $\E$
determines the renewed compatible scalar product $\angles{\cdot, \cdot}_{Q}$
associated to $Q$
and the $(\Z/2\Z)$-grading given by the $\pm 1$-eigen spaces of $U$.
Conversely, if a Hilbert $A$-module $\E$ is equipped with a $(\Z/2\Z)$-grading,
then it determines a regular quadratic form $Q$ given by
$Q(\nu, \xi) = \angles{\nu, (-1)_{}^{\deg(\xi)}\xi }$ for homogeneous elements. 
\end{rem}
\begin{dfn}
\label{DfnJA}
Let $A$ be a $C^{*}$-algebra.
$\mathbb{J}(A)$ denotes the space consisting of unitary equivalent classes of triples
$(\E, Q, \delta)$, where 
$\E$ is a Hilbert $A$-module, $Q$ is a regular quadratic form on $\E$
and $\map {\delta}{\mathrm{dom}(\delta)} {\E}$ is a densely defined closed operator
satisfying the following conditions;
\begin{enumerate}
\item
$\delta_{}' = -\delta$, namely, $Q(-\delta(\nu), \xi) = Q(\nu, \delta(\xi))$ for
$\nu, \xi \in \mathrm{dom}(\delta)$.
\item
$\Im (\delta) \subset \mathrm{dom}(\delta)$ and $\delta_{}^{2} =0$.
\item
There exists $\sigma, \tau \in \mathbb{K}(\E)$ satisfying
$\sigma \delta + \delta \tau - 1 \in \mathbb{K}(\E)$.
\end{enumerate}
\end{dfn}
The typical example, which we will use for dealing with the signature,
is given by Definition \ref{DfnQuadraticForSgn}.
Roughly speaking, $\E$ is a completion of the space of compactly supported 
differential forms $\Omega_{c}^{*}$, $Q$ is given by the Hodge $\ast$-operation
and $\delta$ is the exterior derivative.
\begin{rem}
This definition is slightly different from $\mathbf{L}_{\mathrm{nb}}^{}(A)$
in \cite[1.5 D\'{e}finition]{HilSk92} and our $\mathbb{J}(A)$ is smaller.
However it is sufficient for our purpose.
\end{rem}
\begin{lem}
If a closed operator $\delta$ satisfies the condition (3),
then both operators $(\delta + \delta_{}^{*} \pm i)_{}^{-1}$ can be defined and 
they belong to $\mathbb{K}(\E)$.
Here, $\delta_{}^{*}$ denotes the adjoint of $\delta$ with respect to
a certain scalar product on $\E$.
\end{lem}
\begin{proof}
Since $\delta$ is a closed operator, $\delta + \delta_{}^{\ast}$ is self-adjoint.
Thus $\Im(\delta + \delta_{}^{\ast} \pm i)$ are equal to $\E$
and both operators $\delta + \delta_{}^{\ast} \pm i$ are invertible.
We now claim that 
both $(\delta + \delta_{}^{\ast} \pm i)_{}^{-1} \in \mathbb{L}(\E)$
are compact operators.
Since $\Im((\delta + \delta_{}^{\ast} \pm i)_{}^{-1})=\mathrm{dom}(\delta + \delta_{}^{\ast} \pm i) = \mathrm{dom}(\delta) \cap \mathrm{dom}(\delta_{}^{\ast})$
and $\delta$ and $\delta_{}^{\ast}$ are closed operators,
the following operators
\begin{align*}
\alpha_{\pm}^{} := \delta (\delta + \delta_{}^{\ast} \pm i)_{}^{-1}
\qqand
\beta_{\pm}^{} := \delta_{}^{\ast} (\delta + \delta_{}^{\ast} \pm i)_{}^{-1}
\end{align*}
are closed operator
defined on entire $\E$, which implies that they are bounded;
$\alpha, \beta \in \mathbb{L}(\E)$.

On the other hand,
note that $(\sigma \delta)_{}^{2} = (\sigma \delta)(1-\delta \tau) = \sigma \delta$
and $(\delta \tau)_{}^{2} = \delta \tau$
modulo $\mathbb{K}(\E)$.
Let $p$ be the orthogonal projection onto $\Im(\delta \tau)$
and let $q = 1-p$.
Then we have that $ p (\delta \tau) = \delta \tau$
and $(\delta \tau) p = p $ modulo $\mathbb{K}(\E)$.
Moreover,
\begin{align*}
(\sigma \delta) q &= 
(1 - \delta \tau) (1-p) = 1 - \delta \tau - p + (\delta \tau) p = 1 - \delta \tau = \sigma \delta,
\\
q (\sigma \delta) &= 
(1-p) (1 - \delta \tau) = 1 - p - \delta \tau + p (\delta \tau) = 1 - p = q,
\\
1- (\delta_{}^{\ast} \sigma_{}^{\ast}) q - (\delta \tau) p
&=
1- (q \sigma \delta)_{}^{\ast} - p
\\
&=
1- q_{}^{\ast} - p
=
1- q - p
=
0 \quad \text{ modulo } \mathbb{K}(\E).
\end{align*}
Then, set $\ell := 1- (\delta_{}^{\ast} \sigma_{}^{\ast}) q - (\delta \tau) p \in \mathbb{K}(\E)$.
Now we conclude that
\begin{align*}
1 &= \ell + (\delta_{}^{\ast} \sigma_{}^{\ast} q - \delta \tau p ),
\\
(\delta + \delta_{}^{\ast} \pm i)_{}^{-1}
&= (\delta + \delta_{}^{\ast} \pm i)_{}^{-1}\ell
 + (\alpha_{\mp}^{\ast} \sigma_{}^{\ast} q - \beta_{\mp}^{\ast} \tau p ) 
\in \mathbb{K}(\E)
\end{align*}
because $\ell$, $\sigma$ and $\tau$ belong to $\mathbb{K}(\E)$ and
$\alpha_{\pm}^{} $ and $\beta_{\pm}^{} $ belong to $\mathbb{L}(\E)$.
\end{proof}

\begin{dfn}
For $(\E, Q, \delta) \in \mathbb{J}(A)$,
we define the $K$-homology class $\Psi(\E, Q, \delta) \in K_{0}(A)$ as follows.
As in Lemma \ref{LemAssociatedScalarProd},
let $\E$ be equipped with the compatible scalar product $\angles{\cdot, \cdot}_{Q}^{}$ and 
$(\Z/2\Z)$-grading associated to $Q$.
Next, put
\begin{align*}
F_{\delta}^{} := (\delta + \delta_{}^{\ast})
\parens{1 + (\delta + \delta_{}^{\ast})_{}^{2}}_{}^{-\half} \in \mathbb{L}(\E),
\end{align*}
where $\delta_{}^{*}$ is the adjoint of $\delta$ with respect to the scalar product
$\angles{\cdot, \cdot}_{Q}^{}$.
Obviously $F_{\delta}^{}$ is self-adjoint and 
$F_{\delta}^{}$ is an odd operator since $U \delta U = \delta_{}' = -\delta$.
Moreover it follows that
\begin{align*}
1-F_{\delta}^{2}
=
\parens{1 + (\delta + \delta_{}^{\ast})_{}^{2}}_{}^{-1} \in \mathbb{K}(\E)
\end{align*}
by the previous lemma.
Then we define 
$\Psi(\E, Q, \delta) := (\E, F_{\delta}^{}) \in K{\!}K(\C, A) \cong K_{0}(A)$.
The action of $\C$ on $\E$ is the natural multiplication.
\end{dfn}

\begin{lem}
\label{LemKerEqIm}
For $(\E, Q, \delta) \in \mathbb{J}(A)$ satisfying $\Im (\delta) = \Ker (\delta)$,
$\Psi (\E, Q, \delta) = 0 \in K_{0}(A)$.
\end{lem}
\begin{proof}
First, remark that 
$\Im (\delta)$ and $\Ker (\delta_{}^{\ast})$ are orthogonal to each other,
and hence, $\Im (\delta) \cap \Ker (\delta_{}^{\ast}) = \{ 0 \}$.
Indeed, for $\delta(\eta) \in \Im (\delta)$ and $\nu \in \Ker (\delta_{}^{\ast})$,
it follows that $\angles{\delta(\eta), \nu} = \angles{\eta, \delta_{}^{\ast}(\nu)} = 0$.
Now let $\xi \in \Ker (\delta + \delta_{}^{\ast})$.
Then 
$$
0
= \angles{\xi,\ (\delta + \delta_{}^{\ast})_{}^{2}(\xi)}
= \angles{\xi,\ \delta_{}^{\ast} \delta (\xi) + \delta \delta_{}^{\ast} (\xi)}
= \angles{\delta (\xi), \delta (\xi)} + \angles{\delta_{}^{\ast} (\xi), \delta_{}^{\ast} (\xi)},
$$
which implies that $\xi \in \Ker (\delta) \cap \Ker (\delta_{}^{\ast})
= \Im (\delta) \cap \Ker (\delta_{}^{\ast}) = \{ 0 \}$.
Therefore, $\Ker (F_{\delta}^{}) = \{ 0 \}$. Since $F_{\delta}^{}$ is a bounded self-adjoint operator,
it is invertible.
To conclude, $(\E, F_{\delta}^{}) = 0 \in K{\!}K(\C, A)$.
\end{proof}

\subsection{Perturbation arguments}

\begin{lem}{\upshape \cite[2.1. Lemme]{HilSk92}}
\label{LemHilSkPerturb}
Let $(\E_{X}^{}, Q_{X}^{}, \delta_{X}^{}), (\E_{Y}^{}, Q_{Y}^{}, \delta_{Y}^{}) \in \mathbb{J}(A)$.
Suppose that we have...
\begin{enumerate}
\item
$T \in \mathbb{L}(\E_{X}^{}, \E_{Y}^{})$
satisfying $T(\dom (\delta_{X}^{}) ) \subset \dom(\delta_{Y}^{})$, $T \delta_{X}^{} = \delta_{Y}^{} T$
and $T$ induces an isomorphism $\map{[T]} {\Ker(\delta_{X}^{})/ \Im(\delta_{X}^{})} {\Ker(\delta_{Y}^{})/ \Im(\delta_{Y}^{})}$;
\item
$\phi \in \mathbb{L}(\E_{X}^{})$ satisfying $\phi (\dom (\delta_{X}^{}) ) \subset \dom(\delta_{X}^{})$
and $1-T_{}^{\prime}T = \delta_{X}^{}\phi + \phi \delta_{X}^{}$;
\item
$\ep \in \mathbb{L}(\E_{X}^{})$
satisfying $\ep_{}^{2}=1$, $\ep_{}' = \ep$, $\ep \delta_{X}^{} = - \delta_{X}^{} \ep$
and
$\ep (1-T_{}^{\prime}T) = (1-T_{}^{\prime}T) \ep$.
\end{enumerate}
Then, $\Psi(\E_{X}^{}, Q_{X}^{}, \delta_{X}^{}) = \Psi(\E_{Y}^{}, Q_{Y}^{}, \delta_{Y}^{}) \in K_{0}(A)$.
\end{lem}
\begin{proof}
First, we may assume that $\phi_{}' = -\phi$.
Indeed, since $1- T_{}' T = (1- T_{}' T)_{}' = (\delta_{X}^{}\phi + \phi \delta_{X}^{})_{}' 
= - (\delta_{X}^{}\phi_{}' + \phi_{}' \delta_{X}^{})$,
we may replace $\phi$ by $\half (\phi - \phi_{}')$ which satisfies the same assumption.

Set $\E := \E_{X}^{} \oplus \E_{Y}^{} $, $Q:=Q_{X}^{} \oplus (-Q_{Y}^{})$ and 
$\nabla := \bmatrixtwo{\delta_{X}^{}}{0}{0}{-\delta_{Y}^{}}$.
Note that the replacing of $Q_{Y}^{}$ by $-Q_{Y}^{}$ means the reversing of the grading of $\E_{Y}^{}$.
Then it is easy to see that 
$\Psi(\E, Q, \nabla) = \Psi(\E_{X}^{}, Q_{X}^{}, \delta_{X}^{}) - \Psi(\E_{Y}^{}, Q_{Y}^{}, \delta_{Y}^{})$.
Therefore it is sufficient to verify that $\Psi(\E, Q, \nabla) = 0$.

Let us introduce invertible operators $R_{t}^{} \in \mathbb{L}(\E)$ 
and a quadratic form $B_{t}^{}$ on $\E$ given by the formula:
\begin{align*}
R_{t}^{} := \bmatrixtwo{1}{0}{itT\ep}{1} \qqand
B_{t}^{} (\nu,\ \xi):= Q(R_{t}^{} \nu,\ R_{t}^{} \xi) = Q(R_{t}^{\prime} R_{t}^{} \nu,\ \xi)
\end{align*}
for $t \in [0,1]$. We claim that $(\E, B_{t}^{}, \nabla) \in \mathbb{J}(A)$.

It is easy to see that $\nabla R_{t}^{} = R_{t}^{} \nabla$, and hence,
$
B_{t}^{} (\nu, \nabla \xi) 
= B_{t}^{} (-\nabla \nu, \xi).
$
Clearly the scalar products associated to $B_{t}^{}$ and $Q$ are compatible with each other,
also the condition (2) and (3) in the definition of $\mathbb{J}(A)$ are satisfied.
Therefore $(\E, B_{t}^{}, \nabla) \in \mathbb{J}(A)$ and $\Psi (\E, B_{t}^{}, \nabla) = \Psi (\E, Q, \nabla)$.

Next let us introduce 
\begin{align*}
L_{t}^{} := \bmatrixtwo{1-T_{}^{\prime}T}{(i\ep + t\phi)T_{}^{\prime}}{T(i\ep + t\phi)}{1} \qqand
C_{t}^{} (\nu,\ \xi) := Q(L_{t}^{} \nu,\ \xi).
\end{align*}
Notice that since $Q=Q_{X}^{} \oplus (-Q_{Y}^{})$ and that $T_{}^{\prime}$ denotes the adjoint of $T$
with respect to $Q_{X}^{}$ and $Q_{Y}^{}$, the adjoint of the matrix 
$\bmatrixtwo{0}{0}{T}{0}$ with respect to $Q$ is equal to $\bmatrixtwo{0}{-T_{}^{\prime}}{0}{0}$.
Thus we have that $R_{t}^{\prime} = \bmatrixtwo{1}{it\ep T_{}^{\prime}}{0}{1}$
and that
\begin{align*}
R_{1}^{\prime} R_{1}^{}  = \bmatrixtwo{1-\ep T_{}^{\prime} T \ep}{i\ep T_{}^{\prime}}{iT\ep}{1}
=\bmatrixtwo{\ep(1- T_{}^{\prime} T)\ep}{i\ep T_{}^{\prime}}{iT\ep}{1}
=\bmatrixtwo{(1- T_{}^{\prime} T)\ep_{}^{2}}{i\ep T_{}^{\prime}}{iT\ep}{1}
=L_{0}^{}.
\end{align*}
In particular, $B_{1}^{}= C_{0}^{}$.
Since $L_{t}^{}$ is invertible at $t=0$, there exists $t_{0}^{}>0$ such that 
$L_{t}^{}$ is invertible for $t \in [0, t_{0}^{}]$. 
Besides it is clear that $L_{t}^{\prime} = L_{t}^{}$, 
so $C_{t}^{}$ is a regular quadratic form for $t\in[0, t_{0}^{}]$.

Moreover consider the operator $\nabla_{t}^{}:= \bmatrixtwo{\delta_{X}^{}}{tT_{}^{\prime}}{0}{-\delta_{Y}^{}}$,
and we claim that $(\E, C_{t}^{}, \nabla_{t}^{}) \in \mathbb{J}(A)$.
for $t\in[0, t_{0}^{}]$.
The adjoint of $\nabla_{t}^{}$ with respect to the quadratic form $ C_{t}^{}$
is equal to $L_{t}^{-1} \nabla_{t}^{\prime} L_{t}^{}$ 
so in order to check that it is equal to $-\nabla_{t}^{}$,
we should check that $ L_{t}^{} \nabla_{t}^{} = -\nabla_{t}^{\prime} L_{t}^{}$.

\begin{align*}
L_{t}^{} \nabla_{t}^{}
&=
\bmatrixtwo
{(1-T_{}^{\prime}T)\delta_{X}^{} } 
{t(1-T_{}^{\prime}T)T_{}^{\prime}-(i\ep + t\phi)T_{}^{\prime}\delta_{Y}^{}}
{T(i\ep + t\phi)\delta_{X}^{}} 
{tT(i\ep + t\phi)T_{}^{\prime}-\delta_{Y}^{}}
\\
\nabla_{t}^{\prime} L_{t}^{}
&=
\bmatrixtwo
{-\delta_{X}^{} (1-T_{}^{\prime}T)} 
{-\delta_{X}^{}(i\ep + t\phi)T_{}^{\prime}}
{-tT(1-T_{}^{\prime}T)-\delta_{Y}^{}T(i\ep + t\phi)}
{-tT(i\ep + t\phi)T_{}^{\prime}+\delta_{Y}^{}}
\end{align*}
Obviously the (1,1) and (2,2)-entries are the negative of each other. Besides we can see that
\begin{align*}
[\text{(1,2)-entry of } L_{t}^{} \nabla_{t}^{} ]
&= t(\delta_{X}^{} \phi +\phi \delta_{X}^{})T_{}^{\prime}-(i\ep + t\phi)\delta_{X}^{}T_{}^{\prime}
\\  &= t \delta_{X}^{} \phi T_{}^{\prime}- i\ep \delta_{X}^{}T_{}^{\prime} 
\\  & = \delta_{X}^{} (i\ep + t \phi) T_{}^{\prime}
 =- [\text{(1,2)-entry of } \nabla_{t}^{\prime} L_{t}^{}]. 
\end{align*}
Since $(L_{t}^{} \nabla_{t}^{})_{}^{\prime} = \nabla_{t}^{\prime} L_{t}^{}$, 
it automatically follows that 
$[\text{(2,1)-entry of } L_{t}^{} \nabla_{t}^{} ]
=- [\text{(2,1)-entry of } \nabla_{t}^{\prime} L_{t}^{}]$ as well,
and now we obtained that $ L_{t}^{} \nabla_{t}^{} = -\nabla_{t}^{\prime} L_{t}^{}$.
It is easy to see that $(\nabla_{t}^{})^{2}=0$.
If $\sigma_{X}^{}, \tau_{X}^{} \in \mathbb{K}(\E_{X}^{})$ and 
$\sigma_{Y}^{}, \tau_{Y}^{} \in \mathbb{K}(\E_{Y}^{})$
satisfy
$\sigma_{X}^{} \delta_{X}^{} + \delta_{X}^{} \tau_{X}^{} - 1 \in \mathbb{K}(\E_{X}^{})$
and
$\sigma_{Y}^{} \delta_{Y}^{} + \delta_{Y}^{} \tau_{Y}^{} - 1 \in \mathbb{K}(\E_{Y}^{})$,
then it follows that
$\brackets{ \smatrixtwo{\sigma_{X}^{}}{0}{0}{-\sigma_{Y}^{}} } \nabla_{t}^{} + 
\nabla_{t}^{} \brackets{ \smatrixtwo{\tau_{X}^{}}{0}{0}{-\tau_{Y}^{}} } -1\in \mathbb{K}(\E)$
since $T \in \mathbb{L}(\E_{X}^{}, \E_{Y}^{})$.
Thus we obtained that $(\E, C_{t}^{}, \nabla_{t}^{}) \in \mathbb{J}(\E)$
and $\Psi (\E, C_{t}^{}, \nabla_{t}^{}) = \Psi (\E, B_{1}^{}, \nabla) = \Psi(\E, Q, \nabla)$.

Finally check that $\Ker (\nabla_{t}^{}) = \Im (\nabla_{t}^{})$
for any $t \in (0, t_{0}^{}]$.
$\Ker (\nabla_{t}^{}) \supset \Im (\nabla_{t}^{})$ 
is implied by $(\nabla_{t}^{})^{2}=0$, 
so let $\begin{bmatrix}{\theta_{1}^{}}\\{\theta_{2}^{}}\end{bmatrix} \in \Ker (\nabla_{t}^{})$. 
Then $\theta_{2}^{} \in \Ker (\delta_{Y}^{})$ and $tT_{}^{\prime}\theta_{2}^{} = -\delta_{X}^{}\theta_{1}^{} \in \Im (\delta_{X}^{})$.
Since $T_{}^{\prime}$ induces an isomorphism 
$[T_{}^{\prime}] \colon \Ker (\delta_{Y}^{})/ \Im (\delta_{Y}^{}) \to \Ker (\delta_{X}^{})/ \Im (\delta_{X}^{})$,
it follows from the injectivity that
$\theta_{2}^{} \in \Im (\delta_{Y}^{}) $. There exists $\eta \in \E_{2}^{}$ such that $\delta_{Y}^{} \eta = \theta_{2}^{}$.
On the other hand, 
$\theta_{1}^{} + tT_{}^{\prime} \eta \in \Ker (\delta_{X}^{})$ and the surjectivity of $[T_{}^{\prime}]$ imply that
 there exists $\zeta \in \Ker (\delta_{Y}^{})$ such that $T_{}^{\prime} \zeta = \frac{1}{t}(\theta_{1}^{} + tT_{}^{\prime} \eta)$.
Therefore
$\Im (\nabla_{t}^{}) \ni \nabla_{t}^{} \begin{bmatrix}{0}\\{\zeta-\eta}\end{bmatrix}=
\begin{bmatrix}{tT_{}^{\prime}(\zeta-\eta)}\\{-\delta_{Y}^{}(\eta)}\end{bmatrix}=
\begin{bmatrix}{\theta_{1}^{}}\\{\theta_{2}^{}}\end{bmatrix}$,
which concludes that $\Ker (\nabla_{t}^{}) \subset \Im (\nabla_{t}^{})$.

Due to Lemma \ref{LemKerEqIm}, it follows that 
$\Psi (\E, C_{t}^{}, \nabla_{t}^{}) = 0 \in K{\!}K(\C, A)$ and we conclude that
$\Psi(\E_{X}^{}, Q_{X}^{}, \delta_{X}^{}) - \Psi(\E_{Y}^{}, Q_{Y}^{}, \delta_{Y}^{}) = \Psi(\E, Q, \nabla) = 0$.
\end{proof}

\section{$G$-signature}
\subsection{Description of the Analytic $G$-index}
\settingsGXa
And let $\mathbb{V}$ be a $G$-Hermitian vector bundle over $X$.
In this section, we will define and investigate a $C^{*}(G)$-module denoted by $\E(\mathbb{V})$
obtained by completing $C_{c}^{}(X;\mathbb{V})$.
This will be used for the definition of the index of $G$-invariant elliptic operators,
in particular, the signature operator.
\begin{dfn}{\upshape \cite[Section 5]{Kas16}}
First we define on $C_{c}^{}(X; \mathbb{V})$ the structure of a pre-Hilbert module
over $C_{c}^{}(G)$ using the action of ${G}$ on $C_{c}^{}(X; \mathbb{V})$
given by $\gm[s](x) = \gm(s(\gm_{}^{-1} x))$ for $\gm \in {G}$.
\begin{itemize}
\item
The action of $C_{c}^{}(G)$ on $C_{c}^{}(X; \mathbb{V})$ from the right
is given by
\begin{eqnarray}
\label{ActionOnCalE}
s\cdot b = \int_{{G}} \gm[s] \cdot b(\gm_{}^{-1}) \Delta(\gm)^{-\half} \dif \gm
\in C_{c}^{}(X; \mathbb{V})
\end{eqnarray}
for $s\in C_{c}^{}(X; \mathbb{V})$ and $b\in C_{c}^{}(G)$.
Here, $\Delta$ denotes the modular function.
\item
The scalar product valued in $C_{c}^{}(G)$ is given by
\begin{eqnarray}
\label{ScalarProdOnCalE}
\angles{s_1,s_2}_{\E}^{}(\gm) = \Delta(\gm)^{-\half} \angles{s_1, \gm[s_2]}_{L^{2}(\mathbb{V})}^{}
\end{eqnarray}
for $s_i\in C_c(\mathbb{V})$.
\end{itemize}
Define $\E(\mathbb{V})$ as the completion of $C_c(\mathbb{V})$
in the norm $\norm{\angles{s,s}}_{C^{*}({G})}^{\half}$.
\end{dfn}

\begin{thm}{\upshape \cite[Theorem 5.8]{Kas16}}
\label{ThmKasparovEllipInd}
\settingsGXa
Let
$\map{D} {C_{c}^{\infty}(X;\mathbb{V})}{C_{c}^{\infty}(X;\mathbb{V})}$
be a formally self-adjoint $G$-invariant first-order
elliptic operator on a $G$-Hermitian vector bundle $\mathbb{V}$.
Then both operators $D \pm i$ have dense range as operators
on $\E(\mathbb{V})$ and $(D \pm i)_{}^{-1}$ belong to $\mathbb{K}(\E(\mathbb{V}))$.
The operator $D(1 + D^{2})_{}^{-1/2} \in \mathbb{L}(\E(\mathbb{V}))$
is a Fredholm and determines an element
$\ind_{G}^{}(D) \in 
K_{0}^{}(C^{*}(G))$.
\QED
\end{thm}
In this paper, mainly we consider 
$\mathbb{V}$ as $\turnV{*} T^{*}X$ equipped with the $\Z/2\Z$-grading
given by the Hodge $\ast$-operation and $D$ as a signature operator.

\begin{dfn}
Let $X$ and $Y$ be proper and co-compact Riemannian $G$-manifolds
and let $\mathbb{V}$ and $\mathbb{W}$ be $G$-Hermitian vector bundles over $X$ and $Y$
respectively.
Let $\map{T} {C_{c}^{\infty}(X;\mathbb{V})} {C_{}^{}(Y;\mathbb{W})}$ be a linear operator.
The support of the distributional kernel of $T$ is given by the closure of the complement of
the following union of all subsets $K_{X}^{} \times K_{Y}^{} \subset X \times Y$;
$$
\bigcup_{\substack
{\angles{Ts_{1}^{},\ s_{2}^{}} = 0 \text{ for any sections }
\\
s_{1}^{} \in C_{c}^{}(X; \mathbb{V}) \text{ and }s_{2}^{} \in C_{c}^{}(Y; \mathbb{W}) \text{ satisfying } 
\\
\supp(s_{1}^{}) \subset K_{X}^{},\ \supp(s_{2}^{}) \subset K_{Y}^{}
}}
K_{X}^{} \times K_{Y}^{}.
$$

$T$ is said to be properly supported if both 
$$
\supp(k_{T}^{}) \cup (K_{X}^{} \times Y) \qqand \supp(k_{T}^{}) \cup (X \times K_{Y}^{})
\quad \subset X \times Y
$$
are compact for any compact subset $K_{X}^{} \subset X$ and $K_{Y}^{} \subset Y$.

$T$ is said to be compactly supported if 
$
\supp(k_{T}^{}) \subset X \times Y
$
is compact.
\end{dfn}
The following proposition is used for the construction of the bounded operators
on $\E(\mathbb{V})$.
\begin{prop}{\upshape \cite[Proposition 5.4]{Kas16}}
\label{PropBoundedOpOnEV}
Let $G$, $X$, $Y$, $\mathbb{V}$ and $\mathbb{W}$ be as above.
Let $\map{T} {C_{c}^{}(X;\mathbb{V})} {C_{c}^{}(Y;\mathbb{W})}$
be a properly supported $G$-invariant operator which is $L^{2}$-bounded. 
Then $T$ defines an element of $\mathbb{L}(\E(\mathbb{V}) , \E(\mathbb{W}))$.
\end{prop}
For the proof, we will use the following Lemma \ref{LemCoMo} and 
Lemma \ref{LemPositiveScalarProdOnE}.
\begin{lem}
\label{LemCoMo}
Let $P\in \mathbb{L}(L^{2}(X;\mathbb{V}),L^{2}(Y;\mathbb{W}))$ be 
a compactly supported bounded operator.
Then the operator
$$
\widetilde{P}:=\int_{G} \gm[P] \:\dif \gm
$$ 
is well defined as
a bounded operator in $\mathbb{L}(L^{2}(X;\mathbb{V}),L^{2}(Y;\mathbb{W}))$
and the inequation
$\norm{\widetilde{P}}_{\mathrm{op}} \leq C \norm{P}_{\mathrm{op}}$
holds,
where $C$ is a constant depending on its support.
\end{lem}
\begin{proof}
Assume that the support of the
distributional kernel of $P$ is contained in $K_{X}^{} \times K_{Y}^{}$ for some
compact subsets $K_{X}^{} \subset X$ and $K_{Y}^{} \subset Y$.
We will follow the proof of \cite[Lemma 1.4--1.5]{Co-Mo82}.
Fix an arbitrary smooth section with compact support
$s\in C_c^{\infty} (X;\mathbb{V})$ and let us consider
$F_{s}^{} \in L^{2}\parens{G;L^{2}(Y;\mathbb{W})}$
given by
$$
F_{s}^{}(\gm) := \gm[P]s.
$$
Note that
for any $\gm \in G$
the support of the
distributional kernel of $\gm[P]$ is contained in $\gm(K_{X}^{})\times \gm(K_{Y}^{})$.
This is because for any $s\in C_c^{\infty}(X;\mathbb{V})$, it follows that 
$\supp(\gm[P]s) \subset \gm(K_{Y}^{})$ and  $\gm[P]s=0$
whenever $\supp(s) \cap \gm(K_{X}^{}) = \emptyset$.
In particular, since the actions are proper,
$F_{s}^{}$ has compact support in $G$.
In addition, again since
the actions are proper,
$\gm(K_{Y}^{}) \cap \eta(K_{Y}^{}) = \gm( K_{Y}^{} \cap \gminv \eta (K_{Y}^{}))
= \emptyset$
if $\gminv \eta \in G$ is outside some compact neighborhood
$Z \subset G$ in particular,
$$\norm{F_{s}^{}(\gm)}_{L^{2}(Y;\mathbb{W})}^{}
\cdot \norm{F_{s}^{}(\eta)}_{L^{2}(Y;\mathbb{W})}^{} = 0$$
for such $\gm$ and $\eta \in G$.
Remind that $Z$ is determined only by $K_{Y}^{}$
so independent of $s$.
Then,
\begin{eqnarray*}
\left\Vert
\int_{G}F_{s}^{}(\gm) \:\dif \gm
\right\Vert_{L^{2}(Y;\mathbb{W})}^{2}
&=&
\norm{
\int_{G}F_{s}^{}(\gm) \:\dif \gm
}_{L^{2}(Y;\mathbb{W})}^{}
\norm{
\int_{G}F_{s}^{}(\eta) \:\dif \eta
}_{L^{2}(Y;\mathbb{W})}^{}
\\
&\leq&
\int_{G}\int_{G}
\norm{F_{s}^{}(\gm)}_{L^{2}(Y;\mathbb{W})}^{} \norm{F_{s}^{}(\eta)}_{L^{2}(Y;\mathbb{W})}^{}
 \:\dif \gm  \:\dif \eta
\\
&\leq&
\int_{G} \norm{F_{s}^{}(\gm)}_{L^{2}(Y;\mathbb{W})}^{}
\parens{
\int_{G}
\chi_{Z}^{}(\gminv \eta) \norm{F_{s}^{}(\eta)}_{L^{2}(Y;\mathbb{W})}^{}
 \:\dif \eta } \dif \gm
\\
&\leq&
\norm{F_{s}^{}}_{L^{2}(G)}^{}
\norm{\chi_{Z}^{}}_{L^{2}(G)}^{}
\norm{F_{s}^{}}_{L^{2}(G)}^{}
\\
&\leq&
|Z|\norm{F_{s}^{}}_{L^{2}(G)}^{2},
\end{eqnarray*}
where $\map{\chi_{Z}^{}}{G}{[0,1]}$ is the characteristic function of $C$,
that is $\chi_{Z}^{}(\gm) = 1$ for $\gm \in Z$ and
$\chi_{Z}^{}(\gm) = 0$ for $\gm \notin Z$.

Next, take a compactly supported smooth function
$c_{1}^{} \in C_c^{\infty}(X;[0,1])$ such that $c_{1}^{} = 1$ on $K_{X}^{}$.
Noting that 
$P = P c_{1}^{}$, we obtain
\begin{eqnarray*}
\norm{F_{s}^{}}_{L^{2}(G)}^{2}
&=&
\int_{G}\norm{F_{s}^{}(\gm)}_{L^{2}(Y;\mathbb{W})}^{2} \:\dif \gm
=
\int_{G}\norm{
\gm P c_{1}^{} \gminv s
}_{L^{2}(Y;\mathbb{W})}^{2} \:\dif \gm
\\
&\leq&
\int_{G}\norm{P}_{\mathrm{op}}^{2}\norm{
c_{1}^{} \gminv s
}_{L^{2}(X;\mathbb{V})}^{2} \:\dif \gm
\\
&\leq&
\norm{P}_{\mathrm{op}}^{2}
\int_{G} \int_{X}^{}
|c_{1}^{}(x)|^{2} \norm{\gminv s(x)}_{\mathbb{V}}^{2}
 \:\dif x \:\dif \gm
\\
&\leq&
\norm{P}_{\mathrm{op}}^{2}
\int_{G} \int_{X}^{}
|c_{1}^{}(\gminv x)|^{2} \norm{s(x)}_{\mathbb{V}}^{2}
 \:\dif x \:\dif \gm
\\
&\leq&
\norm{P}_{\mathrm{op}}^{2}
\sup_{x\in X}\parens{\int_{G}
|c_{1}^{}(\gminv x)|^{2} 
\:\dif \gm}
\norm{s}_{L^{2}(X;\mathbb{V})}^{2}.
\end{eqnarray*}
Since the action of $G$ is proper, 
$\set{ \gm \in G }{ \gminv x \in \supp(c_{1}^{})} \subset G$ 
is compact so the value $\int_{G}
|c_{1}^{}(\gminv x)|^{2} 
\:\dif \gm$ is always finite for any fixed $x\in X$.
Besides, since $X/G$ is compact, this value is uniformly bounded;
$$
C:=
\sup_{x\in X}\parens{\int_{G}
|c_{1}^{}(\gminv x)|^{2} 
\:\dif \gm}
=
\sup_{[x]\in X/G}\parens{\int_{G}
|c_{1}^{}(\gminv x)|^{2} 
\:\dif \gm}
< \infty.
$$
Remind that $C$ depends only on $K_{X}^{}$, not on $s$.
We conclude that
$$
\norm{\int_{G}\gm[P]s \:\dif \gm}_{L^{2}(Y;\mathbb{W})}^{2}
=
\norm{\int_{G}F_{s}^{}(\gm) \:\dif \gm}_{L^{2}(Y;\mathbb{W})}^{2}
\leq
|Z|\norm{F_{s}^{}}_{L^{2}(G)}^{2}
\leq
|Z|C \cdot \norm{P}_{\mathrm{op}}^{2}
\norm{s}_{L^{2}(X;\mathbb{V})}^{2}.
$$
\end{proof}
\begin{lem}{\upshape \cite[Lemma 5.3]{Kas16}}
\label{LemPositiveScalarProdOnE}
Let $P$ be a bounded positive operator on $L^{2}(X;\mathbb{V})$
with a compactly supported distributional kernel.
Then the scalar product
$$
(s_1, s_2) \mapsto 
\angles{ s_1,\
 \parens{ \int_{G} \gm[P]\: \dif \gm }s_2
}_{\E(\mathbb{V})} \in C^{*}(G)
$$
is well defined and positive for any
$s_1 = s_2 \in C_{c}(X;\mathbb{V})$.
\end{lem}
\begin{proof}
Note that 
\begin{align*}
\angles{ \gm[s],\ P(\gm[s])}_{L^{2}(X;\mathbb{V})}^{}
=
\angles{ \sqrt{P}(\gm[s]),\ \sqrt{P}(\gm[s])}_{L^{2}(X;\mathbb{V})}^{}
\end{align*}
for $\gm \in G$ and $s \in C_{c}^{}(X;\mathbb{V})$.
Regarding the each side of the above equation as a function in $\gm \in G$,
it is clear that the left hand side vanishes outside some compact subset in $G$
depending on the support of $s$ and $P$.
This implies $\sqrt{P}(\gm[s])$ has a compact support in $G$.
Take any unitary representation space $\mathcal{H}$ of $G$
and $h \in \mathcal{H}$. By the above observation of the compact support,
$$
v := \int_{G} \Delta (\gm)_{}^{-\half} \sqrt{P}(\gm[s]) \otimes \gm [h] \: \dif \gm
\quad \in L_{}^{2}(X; \mathbb{V}) \otimes \mathcal{H}
$$
is well-defined. Then we obtain that
\begin{align*}
0 \leq \norm{v}_{}^{2}
=&
\int_{G}\int_{G} \Delta (\gm)_{}^{-\half} \Delta (\eta)_{}^{-\half}
\angles{\sqrt{P}(\gm[s]),\ \sqrt{P}(\eta[s])}_{ L_{}^{2}(X; \mathbb{V})}^{} 
\angles{\gm [h],\ \eta [h]}_{\mathcal{H}}^{} \: \dif \gm \: \dif \eta
\\
=&
\int_{G}\int_{G} \Delta (\gm)_{}^{-\half} \Delta (\eta)_{}^{-\half}
\angles{s,\ \gminv [P(\eta[s])]}_{ L_{}^{2}(X; \mathbb{V})}^{} 
\angles{h,\ \gminv \eta [h]}_{\mathcal{H}}^{} \: \dif \gm \: \dif \eta
\\
=&
\int_{G}\int_{G} \Delta (\gm)_{}^{-1} \Delta (\gminv \eta)_{}^{-\half}
\angles{s,\ \gminv [P]( \gminv \eta[s])}_{ L_{}^{2}(X; \mathbb{V})}^{} 
\angles{h,\ \gminv \eta [h]}_{\mathcal{H}}^{} \: \dif \gm \: \dif (\gminv \eta) 
\\
=&
\int_{G}\int_{G} \Delta (\zeta)_{}^{-\half}
\angles{s,\ \gminv [P]( \zeta[s])}_{ L_{}^{2}(X; \mathbb{V})}^{} 
\angles{h,\ \zeta [h]}_{\mathcal{H}}^{} \: \dif (\gminv) \: \dif \zeta
\\
=&
\int_{G}\Delta (\zeta)_{}^{-\half}
 \angles{s,\ \parens{\int_{G} \gm [P]\: \dif \gm} ( \zeta[s]) }_{ L_{}^{2}(X; \mathbb{V})}^{} 
\angles{h,\ \zeta [h]}_{\mathcal{H}}^{} \: \dif \zeta
\\
=&
\int_{G}
 \angles{s,\ \parens{\int_{G} \gm [P]\: \dif \gm} (s) }_{ \E(\mathbb{V})}^{} (\zeta)
\cdot \angles{h,\ \zeta [h]}_{\mathcal{H}}^{} \: \dif \zeta
\end{align*}
Recall that the action of 
$f := \angles{s,\ \parens{\int_{G} \gm [P]\: \dif \gm} (s) }_{ \E(\mathbb{V})}^{} \in C_{c}^{}(G)$
on $\mathcal{H}$ is given by $f[h] = \int_{G} f(\zeta) \zeta[h] \: \dif \zeta$ for $h \in \mathcal{H}$.
Thus, by rewriting the above inequality, we have
$\angles{h, f[h]}_{\mathcal{H}^{}} \geq 0$ for any $h$, 
which means that this $f$ is a positive operator
on any unitary representation space $\mathcal{H}$. To conclude, $f$ is positive in $C^{*}(G)$
for any $s \in C_{c}^{}(\E(\mathbb{V}))$.
\end{proof}
\begin{proof}[Proof of Proposition \ref{PropBoundedOpOnEV}]
Let $T_{1}^{} := \half \parens{c T_{}^{*}T + T_{}^{*}T c}$,
which is bounded self-adjoint operator $L^{2}(X;\mathbb{V}) \to L^{2}(X;\mathbb{V})$.
Moreover the distributional kernel of $T_{1}^{}$ is contained in $K\times K$ for
some compact subset $K \subset X$.
By Lemma \ref{LemCoMo}, $\int_{G} \gm[T_{1}^{}]$ is well-defined in $\mathbb{L}( L^{2}(X;\mathbb{V}))$ and 
\begin{align*}
\int_{G} \gm[T_{1}^{}] 
= \int_{G} \half \parens{\gm[c] T_{}^{*}T + T_{}^{*}T \gm[c]}
= T_{}^{*}T.
\end{align*}

Consider a compactly supported continuous function $f \in C_{c}(X; [0,1])$ 
satisfying that $c_{1}^{}=1$ on $K$
so that $c_{1}^{}T_{1}^{}c_{1}^{}=T_{1}^{}$ holds.
Consider the following self-adjoint operator;
$$
P:= c_{1}^{}\parens{ \norm{T}^{2}\norm{c} - T_{1}^{}} c_{1}^{} 
=
c_{1}^{2}\norm{T}^{2}\norm{c} - T_{1}^{}
\in \mathbb{L}(L^{2}(X;\mathbb{V})).
$$
Obviously $P$ is compactly supported and
since $T_{1}^{} \leq \norm{T_{1}^{}} \leq \norm{T}^{2}\norm{c}$, $P$ is positive.
Using Lemma \ref{LemPositiveScalarProdOnE}, 
for any $s \in C_{c}(\mathbb{V})$, the following value is positive;
\begin{align*}
0 &\leq \angles{ s,\
 \parens{ \int_{G} \gm[P]\: \dif \gm }s
}_{\E(\mathbb{V})} 
\\
&\leq
C \norm{T}_{}^{2}\norm{c}\angles{s,s}_{\E(\mathbb{V})}
- \angles{s,\parens{\int_{G} \gm[T_{1}^{}]} s}_{\E(\mathbb{V})}
\in C^{*}(G),
\end{align*}
where $C$ is the maximum of a $G$-invariant bounded function
$\int_{G} \gm[c_{1}^{2}]$, which is independent of $s$.
To conclude,
\begin{align*}
\angles{T(s), T(s)}_{\E(\mathbb{W})}
&= \angles{s,T^{*}T(s)}_{\E(\mathbb{V})}
= \angles{s,\parens{\int_{G} \gm[T_{1}^{}]} s}_{\E(\mathbb{V})}
\\
&\leq C \norm{T}^{2}\norm{c}\angles{s,s}_{\E(\mathbb{V})}.
\end{align*}
\end{proof}

\subsection{Proof of Theorem A}
The theorem we will discuss is the following;
\begin{thmalph}
\label{MainThm}
Let $X$ and $Y$ be oriented even-dimensional complete Riemnannian manifolds and let
a locally compact Hausdorff group $G$ acts on $X$ and $Y$
isometrically, properly and co-compactly.
$\map{f} {Y} {X}$ be a $G$-equivariant orientation preserving homotopy equivalent map.
Let $\partial_{X}^{}$ and $\partial_{Y}^{}$ be the signature operators.
Then $\mathrm{ind}_{G}^{} (\partial_{X}^{}) = \mathrm{ind}_{G}^{} (\partial_{Y}^{} ) \in K_{0}^{}(C^{*}(G))$.
\end{thmalph}

From now on we will slightly change the notation for simplicity.
We will only consider $\mathbb{V}$ for the cotangent bundle
$\turnV{*}T^{*}X \otimes \C$.
Let us use $\E_{X}^{}$ for $\E(\turnV{*}T^{*}X \otimes \C)$.
Let $\Omega_{c}^{*}(X)$ be the space consisting of compactly supported
smooth differential forms on $X$, namely, $C_{c}^{\infty}(X; \mathbb{V})$.
We will prove Theorem \ref{MainThm} using Lemma \ref{LemHilSkPerturb}.

\begin{dfn}
\label{DfnQuadraticForSgn}
Let us introduce the following data $(\E, Q, \delta)$ to present the
$G$-index of the signature operator;
\begin{itemize}
\item
Let $C^{*}(G)$-valued quadratic form $Q_{X}^{}$ be defined by the formula;
\begin{align}
Q_{X}^{}(\nu, \xi)(\gm) := i^{k(n-k)}_{}
 \Delta (\gm)_{}^{-\frac{1}{2}} \int_{X}^{} \bar{\nu} \wedge \gm[\xi] 
\qqfor
\nu \in \Omega_{c}^{k}(X),\ \nu \in \Omega_{c}^{n-k}(X)
,\ \gm \in G,
\end{align}
here $\bar{\nu}$ denotes the complex conjugate.
If $\deg(\nu) + \deg(\xi) \neq \mathrm{dim}(X)$ then $Q_{X}^{}(\nu, \xi) := 0$.
This $\deg$ means the degree of the differential form.
\item
The grading $U_{X}^{}$ determined by $Q_{X}^{}$ is given by
\begin{align}
U_{X}^{}(\xi) = i^{-k(n-k)} \ast \xi
\qqfor \xi \in \Omega_{c}^{k}(X),
\end{align}
where $\ast$ denotes the Hodge $\ast$-operation.

Clearly, $U_{X}^{2}= 1$ and 
$Q_{X}^{}\parens{ \nu,\ U_{X}^{}(\xi) } = \angles{\nu, \xi}_{\E_{X}^{}}$ hold.
\item
$\delta_{X}^{} (\xi):= i^{k}\dif_{X}^{}\xi$
for $\xi \in \Omega_{c}^{k}(X)$, here $\dif_{X}^{}$ denotes the exterior derivative on $X$.
\end{itemize}
We will also use the similar notations for $Y$.
\end{dfn}

\begin{lem}
$(\E_{X}^{}, Q_{X}^{}, \delta_{X}^{}) \in \mathbb{J}(C^{*}(G))$
and $\Psi (\E_{X}^{}, Q_{X}^{}, \delta_{X}^{}) = \mathrm{ind}_{G}^{} (\partial_{X}^{})$,
where $\partial_{X}^{}$ is the signature operator of $X$.
\end{lem}
\begin{proof}
First, obviously $\delta_{}^{2} =0$.
Applying Theorem \ref{ThmKasparovEllipInd} to the signature operator on $X$,
it follows that
$ \delta_{X}^{} - U_{X}^{} \delta_{X}^{} U_{X}^{} \colon \Omega_{c}^{*}(X) \to \E_{X}^{}$
is closable and its closure is self-adjoint.
Let us use $ \delta_{X}^{} - U_{X}^{} \delta_{X}^{} U_{X}^{} $ for also its closure.
Since $\Im (\delta_{X}^{})$ and $\Im ( - U_{X}^{} \delta_{X}^{} U_{X}^{} )$ are orthogonal to each other
with respect to the scalar product $\angles{\cdot, \cdot}_{\E_{X}^{}}$,
it follows that $\delta_{X}^{}$ itself is a closed operator on $\E$.
Moreover, set
$\sigma = \tau := \frac{ \delta_{X}^{\ast} } { 1+ (\delta_{X}^{\ast} + \delta_{X}^{})_{}^{2} }$.
They belong to $\mathbb{K}(\E_{X}^{})$ since
$\frac{ \delta_{X}^{\ast} } { \delta_{X}^{\ast} + \delta_{X}^{} \pm i} \in \mathbb{L}(\E_{X}^{})$
and $\frac{ 1} { \delta_{X}^{\ast} + \delta_{X}^{} \pm i} \in \mathbb{K}(\E_{X}^{})$.
Then from Theorem \ref{ThmKasparovEllipInd}, we obtain 
$$
\sigma \delta_{X}^{} + \delta_{X} \tau - 1 
= \frac{-1} { 1+ (\delta_{X}^{\ast} + \delta_{X}^{})_{}^{2} } \in \mathbb{K}(\E_{X}^{}).
$$
Therefore, $(\E_{X}^{}, Q_{X}^{}, \delta_{X}^{}) \in \mathbb{J}(C^{*}(G))$
and $\Psi (\E_{X}^{}, Q_{X}^{}, \delta_{X}^{}) = \mathrm{ind}_{G}^{} (\partial_{X}^{})$
by the definition of $\Psi$.
\end{proof}

Let $\map{f} {Y} {X}$ be a $G$-equivariant proper orientation preserving homotopy equivalent map
between $n$-dimensional proper co-compact Riemannian $G$-manifolds.
In order to construct a map $T\in \mathbb{L} (\E_{X}^{}, \E_{Y}^{})$
satisfying the hypothesis of Lemma \ref{LemHilSkPerturb},
it is sufficient to construct an $L^{2}$-bounded $G$-invariant operator
$\map{T} {\Omega_{c}^{*}(X)} {\Omega_{c}^{*}(Y)}$
due to Proposition \ref{PropBoundedOpOnEV}.

\begin{rem}
Note that
$\map {f^{*}} {\Omega_{c}^{*}(X)} {\Omega_{c}^{*}(Y)}$ may not be $L^{2}$-bounded 
unless $\map{f}{Y}{X}$ is submersion. 
For instance, let $Y=X=[-1,1]$ and $f(y)=y^{3}$.
Consider an $L^{2}$-form $\omega$ on $X$ given by 
$\omega(x)= \frac{1}{|x|^{1/4}}$.
Actually $\norm{\omega}_{L^{2}(X)}^{2} = \int_{-1}^{1} \frac{1}{|x|^{1/2}} \dif x = 2$, 
however, 
$\norm{f^{*}\omega}_{L^{2}(Y)}^{2} = \int_{-1}^{1}\frac{1}{|y|^{3/2}} \dif y
= +\infty$.
So we need to replace $f^{*}$ by a suitable operator.
\end{rem}
\newcommand{\Wund}{W}
Let us construct operator $T$ that we need
and investigate its properties in a slightly more general condition.
\begin{itemize}
\item
$X$ and $Y$ are Riemannian manifold and $G$ acts on them isometrically and properly.
For a while, $X$ and $Y$ may have boundary and the action may 
not be co-compact if not mentioned.
\item
Let $\Wund$ be an oriented $G$-invariant fiber bundle over $Y$
whose typical fiber is an even dimensional unit open disk $B^{k} \subset \R_{}^{k}$.
Let $q \colon \Wund \twoheadrightarrow Y$ denote the canonical projection map
and $\map{q_{I}^{}} {\Omega_{c}^{* +k}(\Wund)} {\Omega_{c}^{*}(Y)}$ be the 
integration along the fiber.
\item
%
Let us fix $\omega \in \Omega_{}^{k}(\Wund)$ be a $G$-invariant closed $k$-form
with fiber-wisely compact support
such that the integral along the fiber is always equal to $1$;
$q_{I}^{}(\omega)(y) = \int_{\Wund_{y}^{}} \omega = 1$ for any $y \in Y$.
Let $e_{\omega}^{}$ denote the operator 
given by $e_{\omega}^{}(\zeta) = \zeta \wedge \omega$ for $\zeta \in \Omega_{}^{*}(\Wund)$.

We can construct a $G$-invariant $\omega$ as follows;
Let $\tau \in \Omega_{}^{k}(\Wund)$ be a $k$-form inducing a Thom class of $\Wund$.
We may assume that $\int_{\Wund_{y}^{}} \tau = 1$ for any $y \in Y$.
Then $\omega := \int_{G} \gm[c \tau] \dif \gm$ is
a desired $G$-invariant form.
\item
Suppose that we have a $G$-equivariant submersion
$\map{p} {\Wund} {X}$ whose restriction on $\supp(\omega) \subset \Wund$ is proper.
\end{itemize}
\begin{dfn}
For the above data, let us set
$T_{p,\omega}^{} := q_{I}^{} e_{\omega}^{} p^{*} \colon \Omega_{c}^{*}(X) \to \Omega_{c}^{*}(Y)$.
We may write just $T_{p}$ for simplicity.
\begin{align*}
\xymatrix{
\Wund \ar @{>>}[d]_{q} \ar[rd]^{p} &
\\ 
Y  & X
}
\Omega_{c}^{*}(X) 
\underset{p_{}^{*}}{\to} \Omega_{}^{*}(\Wund) 
\underset{e_{\omega}^{}}{\to} \Omega_{c}^{*+k}(\Wund) 
\underset{q_{I}^{}}{\to} \Omega_{c}^{*}(Y).
\end{align*}
\end{dfn}
\begin{lem}
\label{LemEboundedness}
If the actions of $G$ are co-compact, 
then $T_{p,\omega}^{}$ determines an operator in $\mathbb{L}(\E_{X}^{}, \E_{Y}^{})$.
\end{lem}
\begin{proof}
By Proposition \ref{PropBoundedOpOnEV}, it is sufficient to check that 
$T_{p,\omega}^{}$ is $L^{2}$-bounded.

Since $q_{I}^{}$ is obviously $L^{2}$-bounded, only the boundedness of 
$e_{\omega}^{} p^{*} \colon \Omega_{c}^{*}(X) \to \Omega_{c}^{*}(\Wund)$
is non-trivial.
Note that our proper submersion $p$ restricted on $\supp(\omega) \subset \Wund$ is locally trivial 
$G$-invariant fibration. Let $p_{I}^{}$ denotes the integration along this fibration. Then 
$$
\int_{\Wund} \zeta
=
\int_{X} p_{I}^{}\zeta
$$
holds for any compactly supported differential form
$\zeta \in \Omega_{c}^{*}(\Wund)$ satisfying $\supp(\zeta) \subset \supp(\omega)$,
in particular, $\zeta = |(p^{*}\xi)\wedge \omega|_{}^{2}\mathrm{vol}_{\Wund}^{}
 \in \Omega_{c}^{n+k}(\Wund)$
for $\xi \in \Omega_{c}^{*}(X)$.
Let $C_{\omega}^{}$ be the maximum of the norm of bounded $G$-invariant form 
$p_{I}^{}\parens{ |\omega|_{}^{2}\mathrm{vol}_{\Wund}^{} } \in \Omega_{}^{n}(X)$.
\begin{align*}
\norm{e_{\omega}^{} p^{*} (\xi)}_{L^{2}(\Wund)}^{2}
&=
\int_{\Wund} |(p^{*}\xi)\wedge \omega|_{}^{2}\mathrm{vol}_{\Wund}^{}
\overset{(\dagger)}{=}
\int_{X} |\xi|_{}^{2} p_{I}^{}\parens{ |\omega|_{}^{2}\mathrm{vol}_{\Wund}^{} }
\\
&\leq
C_{\omega}^{} \int_{X} |\xi|_{}^{2}\mathrm{vol}_{X}^{}
=
C_{\omega}^{} \norm{\xi}_{L^{2}(X)}^{2}
\qqfor \xi \in \Omega_{c}^{*}(X).
\end{align*}
The equation $(\dagger)$ holds because the function $p^{*}|\xi|_{}^{2}$ is constant
along the fiber $p^{-1}(x)$.
\end{proof}
\newcommand{\Vund}{V}
\begin{lem}
\label{LemCompositionOfT}
Let us consider proper co-compact $G$-manifold $X$, $Y$ and $Z$
and let $q_{1}^{} \colon \Wund \twoheadrightarrow Y$
and $q_{2}^{} \colon \Vund \twoheadrightarrow Z$ 
be $G$-invariant oriented disk bundles over $Y$ and $Z$
with typical fiber $B^{k_{1}}$ and $B^{k_{2}}$.
Fix $G$-invariant closed forms $\omega_{1}^{} \in \Omega_{}^{k_{1}}(\Wund)$
and $\omega_{2}^{} \in \Omega_{}^{k_{2}}(\Vund)$
with fiber-wisely compact support
satisfying $(q_{j}^{})_{I}^{}(\omega_{j}^{}) = 1$.
Let $\map{p_{1}^{}} {\Wund} {X}$ $\map{p_{2}^{}} {\Vund} {Y}$ be 
$G$-equivariant submersions whose restriction on $\supp(\omega_{j}^{})$ are proper.

On the other hand, as in the diagram below,
 let us consider the pull-back bundle
$p_{2}^{*}\Wund = \set{ (v,w) \in \Vund \times \Wund} { p_{2}^{}(v) = q_{1}^{}(w) }$ over $\Vund$
and let us regard it as a fiber bundle over $Z$ with projection denoted by $q_{21}^{}$.
Let us set
$\omega_{21}^{} := \widetilde{p_{2}^{}}_{}^{*}\omega_{1}^{} \wedge \widetilde{q_{1}^{}}_{}^{*}\omega_{2}^{}
\in \Omega_{}^{*}(p_{2}^{*} \Wund)$,
$p_{21}^{} := p_{1}^{}\widetilde{p_{2}^{}}$, where 
$\map{\widetilde{q_{1}^{}} } {p_{2}^{*}\Wund} {\Vund}$ denotes the projection
and $\map {\widetilde{p_{2}^{}} } {p_{2}^{*}\Wund} {\Wund}$ 
denotes the map induced by $p_{2}^{}$.

Then $T_{p_{2}}^{}T_{p_{1}}^{} = T_{p_{21}}^{} \colon \E_{X}^{} \to \E_{Z}^{}$.
\begin{align*}
\xymatrix{
&&
\\
 \Vund \ar @{>>}[d]_{q_{2}^{}} \ar[rd]|-{p_{2}^{}} 
& \Wund \ar @{>>}[d]^{q_{1}^{}} \ar[rd]|-{p_{1}^{}} &
\\ 
 Z & Y  & X
}
&
\xymatrix{
 p_{2}^{*}\Wund \ar @{>>}[d]^{ \widetilde{q_{1}^{}} } \ar @/^1pc/[rrdd]^{p_{21}^{}} 
 \ar[rd]|-{ \widetilde{p_{2}^{}} } \ar @/_/@{>>}[dd]_{q_{21}^{}}
&&
\\ 
 \Vund \ar @{>>}[d] \ar[rd]
& \Wund \ar @{>>}[d] \ar[rd]&
\\ 
 Z&Y & X
}
\\
T_{p_{2}}^{}T_{p_{1}}^{}
\colon \E_{X}^{} \to \E_{Y}^{} \to \E_{Z}^{},
&
\qquad
T_{p_{21}}^{}
\colon \E_{X}^{} \to \E_{Z}^{}.
\end{align*}
\end{lem}
\begin{proof}
First we can see that for $\xi \in \Omega_{c}^{*}(X)$,
\begin{align*}
T_{p_{21}}^{}(\xi)
&=
(q_{21}^{})_{I}^{} \circ e_{\omega_{21}^{}}p_{21}^{*}(\xi) 
\\
&=  (q_{2}^{})_{I}^{}(\widetilde{q_{1}^{}})_{I}^{}
\braces{
\widetilde{p_{2}^{}}_{}^{*} p_{1}^{*}\xi
\wedge (\widetilde{p_{2}^{}}_{}^{*}\omega_{1}^{} \wedge \widetilde{q_{1}^{}}_{}^{*}\omega_{2}^{})
}
\\
&=(q_{2}^{})_{I}^{}(\widetilde{q_{1}^{}})_{I}^{}
\braces{
\widetilde{p_{2}^{}}_{}^{*}(p_{1}^{*}\xi \wedge \omega_{1}^{}) \wedge \widetilde{q_{1}^{}}_{}^{*}\omega_{2}^{}
}
\\
&=
(q_{2}^{})_{I}^{}\braces{
(\widetilde{q_{1}^{}})_{I}^{}\parens{
\widetilde{p_{2}^{}}_{}^{*}(p_{1}^{*}\xi \wedge \omega_{1}^{}) 
}\wedge \omega_{2}^{}
}
\\
&=
(q_{2}^{})_{I}^{} e_{\omega_{2}^{}}^{} (\widetilde{q_{1}^{}})_{I}^{} \widetilde{p_{2}^{}}_{}^{*} e_{\omega_{1}^{}}^{} p_{1}^{*}(\xi),
\\
T_{p_{2}}^{}T_{p_{1}}^{}(\xi)
&=(q_{2}^{})_{I}^{} e_{\omega_{2}^{}}^{} p_{2}^{*} \circ (q_{1}^{})_{I}^{} e_{\omega_{1}^{}}^{} p_{1}^{*}(\xi).
\end{align*}
Note that $(\widetilde{q_{1}^{}})_{I}^{}$ in the second bottom row is well defined 
because the differential form $\widetilde{p_{2}^{}}_{}^{*} e_{\omega_{1}^{}}^{} p_{1}^{*}(\xi)$ is
compactly supported along each fiber of 
$\widetilde{q_{1}^{}} \colon p_{}^{*}\Wund \twoheadrightarrow \Vund$.
We need to prove the commutativity of the following diagram;
\begin{align}
\xymatrix{
 \Omega_{}^{*}(p_{2}^{*}\Wund) \ar[d]_{ (\widetilde{q_{1}^{}})_{I}^{} }
&
\\ 
 \Omega_{}^{*}(\Vund)
& \Omega_{c}^{*}(\Wund) \ar [d]^{ (q_{1}^{})_{I}^{} } \ar[lu]_{\widetilde{p_{2}^{}}_{}^{*}}
\\ 
 &\Omega_{}^{*}(Y)\ar[lu]_{p_{2}^{*}}
}
\label{DiagramCommqIpstar}
\end{align}
It is easy to check this using local trivializations.
Suppose that $\Wund \twoheadrightarrow Y$ is trivialized on $U \subset Y$.
Then $p_{2}^{*}\Wund$ is trivialized on $p_{2}^{-1}U \subset \Vund$.
We write these trivialization as $\Wund|_{U}^{} \simeq U \times B_{}^{k}$ and
$p_{2}^{*}\Wund|_{U}^{} \simeq p_{2}^{-1}U \times B_{}^{k}$.
Then for $\zeta(y,w) = f(y,w) \dif y \wedge \dif w \in \Omega_{c}^{*}(\Wund|_{U}^{})$,
$$
\parens{ (\widetilde{q_{1}^{}})_{I}^{}\widetilde{p_{2}^{}}_{}^{*} \zeta }(v) 
= \int_{B_{}^{k}}\parens{ f(p_{2}^{}(v),w) p_{2}^{*}(\dif y) }\dif w
= \parens{ p_{2}^{*}(q_{1}^{})_{I} \zeta } (v) \qqfor v \in p_{}^{-1}U \subset \Vund.
$$
\end{proof}
We will use the following proposition repeatedly.
\begin{prop}
\label{KeyPropChainHomotopy}
Let $\Wund_{1}^{}$ and $\Wund_{2}^{}$ be oriented $G$-invariant disk bundles over $Y$
with typical fiber $B^{k_{1}}$ and $B^{k_{2}}$,
and let $q_{j}^{} \colon W_{j}^{} \twoheadrightarrow Y$ be the projection.
Let $\omega_{j}^{} \in \Omega_{}^{k_{j}}$ be closed  forms
with fiber-wisely compact support
satisfying $(q_{j}^{})_{I}^{}(\omega) \omega_{j}^{} =1$.

Suppose that there exist $G$-equivariant submersions 
$\map{p_{j}^{} } {\Wund_{j}^{}} {X}$ whose restriction on the $0$-sections
$\map{p_{j}^{}( \cdot, 0) } {Y} {X}$ are $G$-equivariant homotopic
to each other.

Then, there exists a properly supported $G$-equivariant $L_{}^{2}$-bounded operator 
$\psi \colon \Omega_{c}^{*}(X) \to \Omega_{c}^{*}(Y)$ satisfying that
$T_{p_{2}^{},\omega_{2}^{}}^{} - T_{p_{1}^{},\omega_{1}^{}}^{} = \dif_{X}^{} \psi + \psi \dif_{Y}^{}$.
\end{prop}

First, let us prove the following lemma;
\begin{lem}
\label{LemChainHomotopy}
Let $Q \colon \widetilde{\Wund} \twoheadrightarrow Y\times[0,3]$
be a $G$-invariant disk bundle over $Y\times[0,3]$
and let $\omega \in \Omega_{}^{k}(\widetilde{\Wund})$ be a closed form 
with fiber-wisely compact support
satisfying $Q_{I}^{} (\omega) =1$.
Suppose that there exists a $G$-equivariant submersion
$P \colon \widetilde{\Wund} \to X$ whose restriction on $\supp(\omega)$ is proper.
Then there exists a properly supported $G$-equivariant $L_{}^{2}$-bounded operator 
$\psi \colon \Omega_{c}^{*}(X) \to \Omega_{c}^{*}(Y)$ satisfying that
$T_{P(\cdot,3), \omega(\cdot,3)}^{}- T_{P(\cdot,0), \omega(\cdot,0)}^{} =\dif_{X}^{} \psi + \psi \dif_{Y}^{}$.
\end{lem}
\begin{proof}
Let $\xi \in \Omega_{c}^{*}(X)$
and $\theta := Q_{I}^{}(P^{*}\xi \wedge \omega) \in \Omega_{c}^{*}(Y \times [0,3])$.
Then it is easy to see that 
\begin{align*}
\int_{[0,3]}\dif \theta = -\dif
 \parens{ \int_{[0,3]} \theta } + \parens{ i_{3}^{*}\theta - i_{0}^{*}\theta },
\end{align*}
where $i_{t}^{} \colon Y \times \{ t \} \hookrightarrow Y \times [0,3]$
denotes the inclusion map.
Note that $i_{t}^{*}\theta = T_{P(\cdot, t), \omega(\cdot, t)}^{} \xi$.

Now, set $\psi \colon \Omega_{c}^{*}(X) \to \Omega_{c}^{*}(Y)$
 by the formula; $\psi(\xi) := \int_{[0,3]} Q_{I}^{}(P^{*}\xi \wedge \omega)$
for $\xi \in \Omega_{c}^{*}(X)$.
Note that the identity map $L^{1}([0,3]) \to L^{2}([0,3])$ is a continuous inclusion
due to the finiteness of $\mathrm{vol}([0,3])$
hence, the map $\int_{[0,3]} \colon \Omega_{c}^{*}(Y \times [0,3]) \to \Omega_{c}^{*}(Y)$
is $L^{2}$-bounded.
Moreover, since $P^{*}\xi \wedge \omega$
vanishes at the boundary of each fiber of $\widetilde{\Wund}$,
the integration along the fiber commutes with taking exterior derivative,
in particular, 
$$
\dif \theta = \dif Q_{I}^{}(P^{*}\xi \wedge \omega) = Q_{I}^{} \dif (P^{*}\xi \wedge \omega) = 
Q_{I}^{} (P^{*} (\dif \xi) \wedge \omega).
$$
To  conclude, we obtain
$$
\psi (\dif \xi) = \int_{[0,3]} \dif Q_{I}^{}(P^{*}\xi \wedge \omega)
= -\dif \psi (\xi) + T_{P(\cdot, 3), \omega(\cdot, 3)}^{} \xi - T_{P(\cdot, 0), \omega(\cdot, 0)}^{} \xi,
$$
\end{proof}
\begin{proof}[Proof of Proposition \ref{KeyPropChainHomotopy}]
We need to construct $\widetilde{\Wund}$ and $P$ as above satisfying
$T_{P(\cdot,0), \omega(\cdot,0)}^{} = T_{p_{1}^{},\omega_{1}^{}}^{}$
and $T_{P(\cdot,3), \omega(\cdot,3)}^{} = T_{p_{2}^{},\omega_{2}^{}}^{}$.

Let $h \colon Y \times [0,3] \to X$ be a re-parametrized $G$-homotopy between
$p_{1}^{}( \cdot, 0)$ and $p_{2}^{}( \cdot, 0)$,
that is, 
$h$ is a $G$-equivariant smooth map satisfying
\begin{align*}
h(y,t) &= p_{1}^{}(y, 0) \qqfor t\in [0,1]
\\
\qqand
h(y,t) &= p_{2}^{}(y, 0) \qqfor t\in [2,3].
\end{align*}
here $G$ acts on $[0,3]$ trivially.
Moreover, consider the following fiber product
$\Wund_{1}^{} \times_{Y}^{} \Wund_{2}^{} =
 \set{(y_{1}^{}, w_{1}^{}),(y_{2}^{}, w_{2}^{}) \in \Wund_{1}^{} \times \Wund_{2}^{} }
 {y_{0}^{} = y_{1}^{}}$.
Let us introduce a smooth map $\map{\chi}{[0,3]}{[0,1]}$
satisfying that
\begin{align*}
\chi(t) &= 0 \qqfor t \in \left[ 0,\tfrac{1}{10} \right) \cup \left( \tfrac{29}{10}, 3 \right]
\\
\qqand
\chi(t) &= 1 \qqfor \quad t \in \parens{ \tfrac{9}{10}, \tfrac{21}{10} }.
\end{align*}
Then 
\begin{align*}
\widetilde{h} \colon \parens{\Wund_{1}^{} \times_{Y}^{} \Wund_{2}^{} }\times [0,3]
\to & X
\\
((y,t),w_{1}^{},w_{2}^{}) \mapsto &
    \begin{cases}
      p_{1}^{}(y, (1-\chi(t))w_{1}^{})
 & \qqfor t\in [0,1],\\
      h(y,t)
 & \qqfor t\in [1,2], \\
      p_{2}^{}(y, (1-\chi(t))w_{2}^{})
 & \qqfor t\in [2,3].
    \end{cases}
\end{align*}
This $\widetilde{h}$ is submersion as long as $\chi(t) \neq 1$
due to the submergence of $p_{1}^{}$ and $p_{2}^{}$.
Let $BX:= \set{ v \in TX}{\norm{v} < 1}$ be the unit disk tangent bundle
and consider the pull-back bundle
$ \widetilde{\Wund}:= \widetilde{h}_{}^{*} BX$
and let us regard it as a bundle over $Y \times [0,3]$
and set 
\begin{align*}
P \colon \quad \widetilde{\Wund} \to& X
\\
((y,t), w_{1}^{}, w_{2}^{}, v) \mapsto &
\exp_{\widetilde{h}((y,t),w_{1}^{},w_{2}^{}) }^{}(\chi(t) v).
\end{align*}
Due to the $(\chi(t) v)$-component, $P$ is submersion
also when $\chi(t) \neq 0$ not only when $\chi(t) \neq 1$.

Moreover, define $\omega \in \Omega_{}^{*}(\Wund)$ as 
$\omega := \pi_{1}^{*} \omega_{1}^{} \wedge \pi_{2}^{*} \omega_{2}^{} \wedge 
\widetilde{h}_{}^{*} \omega_{BX}^{}$,
where $\pi_{j}^{} \colon \widetilde{\Wund} \twoheadrightarrow \Wund_{j}$
for $j=1,2$ and $\omega_{BX}^{} \in \Omega_{}^{*}(BX)$ is a $G$-invariant
differential with fiber-wisely compact support satisfying $\int_{BX_{x}^{}} \omega_{BX}^{} =1$.
These $\widetilde{\Wund}$, $P$ and $\omega$ satisfy the assumption of Lemma \ref{LemChainHomotopy}.

It is easy to see that
$T_{P(\cdot,0), \omega(\cdot,0)}^{} = T_{p_{1}^{},\omega_{1}^{}}^{}$
and $T_{P(\cdot,3), \omega(\cdot,3)}^{} = T_{p_{2}^{},\omega_{2}^{}}^{}$
as follows.
For the simplicity, let 
$\pi \colon \widetilde{\Wund}_{Y \times \{0\}} \twoheadrightarrow \Wund_{1}^{}$
denote the projection.
Note that
$P(y,0) = p_{1}^{} \pi $ and we can write
$\omega (\cdot,0) = \pi_{}^{*}\omega_{1}^{} \wedge \widetilde{\omega}$,
using some $\widetilde{\omega} \in \Omega_{}^{*}(\widetilde{\Wund}_{Y \times \{0\}} )$ 
satisfying $\pi_{I}^{} \widetilde{\omega} = 1$.
Then we obtain that
\begin{align*}
T_{P(\cdot,0), \omega(\cdot,0)}^{}(\xi)
&= (q_{1}^{})_{I}^{}\pi_{I}^{} (\pi_{}^{*} p_{1}^{*} \xi \wedge \pi_{}^{*}\omega_{1}^{} \wedge \widetilde{\omega})
\\
&= (q_{1}^{})_{I}^{}\pi_{I}^{} (\pi_{}^{*} ( p_{1}^{*} \xi \wedge \omega_{1}^{})  \wedge \widetilde{\omega})
\\
&= (q_{1}^{})_{I}^{} (( p_{1}^{*} \xi \wedge \omega_{1}^{})  \wedge \pi_{I}^{} \widetilde{\omega})
\\
&= (q_{1}^{})_{I}^{} (p_{1}^{*} \xi \wedge \omega_{1}^{})
\qquad = \quad T_{p_{1}^{},\omega_{1}^{}}^{}(\xi),
\end{align*}
and similarly, $T_{P(\cdot,3), \omega(\cdot,3)}^{} = T_{p_{2}^{},\omega_{2}^{}}^{}$.
\end{proof}
Now let us define a map $T \in \mathbb{L}(\E_{X}^{}, \E_{Y}^{})$
which satisfies the assumption of Lemma \ref{LemHilSkPerturb}.
First, remark that our map $\map{f}{Y}{X}$ is a proper map
by Lemma \ref{LemPropernessOfEquivMap}.
\begin{dfn}
\label{ConstructionOfTf}
Let $BX:= \set{ v \in TX}{\norm{v} < 1}$ be the unit disk tangent bundle
and let $\Wund:=f^{*}BX$ be the pull-back on $Y$, that is,
$\Wund = \set{(y,v) \in Y \times BX} {v \in BX|_{f(y)}^{}}$.
Let $\widetilde{f} \colon \Wund \to BX$ be 
a map given by $\widetilde{f} (x,v) := (f(x),v) $.
Since the action of $G$ on $X$ is isometric
and $f$ is $G$-equivariant, $G$ acts on $BX$ and also on $\Wund$.
Consider a $G$-equivariant submersion given by the formula;
\begin{align}
p \colon \quad \Wund & \to X
\nonumber\\
(y,v) & \mapsto \exp_{f(y)}^{}(v). 
\end{align}
Let us fix a $G$-invariant $\R$-valued closed $n$-form $\omega_{0}^{} \in \Omega^{n}(BX)$
with fiber-wisely compact support
whose integral along the fiber is always equal to $1$,
and let $\omega := \widetilde{f}_{}^{*}\omega_{0}^{} \in \Omega^{n}(\Wund)$
For these $\Wund$, $p$ and $\omega$, let us set $T := T_{p, \omega}^{}$.
\end{dfn}
\begin{lem}
The adjoint with respect to quadratic forms $Q_{X}^{}$ and $Q_{Y}^{}$
is given by $T_{}'= p_{I}^{} e_{\omega}^{}q_{}^{*}$.
\end{lem}
\begin{proof}
Note that $\deg (\omega) = \dim (X)$ is even, hence,
$\omega$ commutes with other differential forms.
For $\nu \in \Omega_{c}^{k}(Y)$ and $\xi \in \Omega_{c}^{n-k}(X)$,
\begin{align*}
\int_{X} p_{I}^{} e_{\omega}^{}q_{}^{*}(\nu) \wedge \xi
&=
\int_{X} p_{I}^{}\parens{ q_{}^{*}\nu \wedge \omega } \wedge \xi
=
\int_{X} p_{I}^{}\parens{ q_{}^{*}\nu \wedge \omega \wedge p_{}^{*} \xi }
=
\int_{BX} q_{}^{*}\nu \wedge \omega \wedge p_{}^{*} \xi
\\
&=
\int_{Y} q_{I} \parens{ q_{}^{*}\nu \wedge p_{}^{*} \xi \wedge \omega }
=
\int_{Y} \nu \wedge q_{I} \parens{ p_{}^{*} \xi \wedge \omega }
=
\int_{Y} \nu \wedge T(\xi).
\end{align*}
Since $Q_{X}^{}(\nu, \xi)(\gm) := i^{k(n-k)}_{}
 \Delta (\gm)_{}^{-\half} \int_{X} \bar{\nu} \wedge \gm[\xi]$,
the proof complete replacing $\nu$ and $\xi$ by $\bar{\nu}$ and $\gm[\xi]$
respectively and using the $G$-invariance of $T$.
\end{proof}
\begin{prop}
\label{PropAssumptionTwo}
There exists $\phi \in \mathbb{L}(\E_{X}^{})$ such that
$1-T_{}'T = \dif_{X}^{}\phi + \phi \dif_{X}^{}$.
\end{prop}
\begin{proof}
Consider the fiber product $\Wund \times_{Y}^{} \Wund$
and let $q_{1}^{}$ and $q_{2}^{} \colon \Wund \times_{Y}^{} \Wund \to \Wund$ 
denote the projections given by
$q_{j}^{}(y, v_{1}^{}, v_{2}^{}) := (y, v_{j}^{})$.
Take $\zeta \in \Omega_{c}^{*}(\Wund)$, here $\Wund$ is regarded as 
the first component of $\Wund \times_{Y}^{} \Wund$.
Using the commutativity of the diagram (\ref{DiagramCommqIpstar}),
\begin{align*}
\xymatrix{
 \Omega_{}^{*}(\Wund \times_{Y}^{} \Wund) \ar[d]_{ (q_{2}^{})_{I}^{} }
&
\\ 
 \Omega_{}^{*}(\Wund)
& \Omega_{c}^{*}(\Wund) \ar [d]^{ q_{I}^{} } \ar[lu]_{ q_{1}^{*}}
\\ 
 &\Omega_{}^{*}(Y)\ar[lu]_{q_{}^{*}}
}
\end{align*}
we obtain that
\begin{align*}
e_{\omega}^{}q_{}^{*} q_{I}^{} (\zeta)
&= e_{\omega}^{} (q_{2}^{})_{I}^{}q_{1}^{*} (\zeta)
= (q_{2}^{})_{I}^{} \parens{q_{1}^{*} \zeta} \wedge \omega
= (q_{2}^{})_{I}^{} \parens{q_{1}^{*} \zeta \wedge q_{2}^{*} \omega}
\\
&= (q_{2}^{})_{I}^{} e_{ q_{2}^{*} \omega }^{} q_{1}^{*}(\zeta),
\\
\text{ and hence, } \quad
T_{}^{\prime}T
=
p_{I}^{} e_{\omega}^{}q_{}^{*} q_{I}^{} e_{\omega}^{} p_{}^{*}
&=
p_{I}^{} (q_{2}^{})_{I}^{} e_{ q_{2}^{*} \omega }^{} q_{1}^{*} e_{\omega}^{} p_{}^{*}.
\end{align*}
On the other hand, since $q_{1}^{}(y,0) = q_{2}^{}(y,0)$,
by Proposition \ref{KeyPropChainHomotopy},
there exists a properly supported
$G$-equivariant $L_{}^{2}$-bounded operator 
$\psi_{\Wund}^{} \colon \Omega_{c}^{*}(\Wund) \to \Omega_{c}^{*}(\Wund)$ satisfying
$$
(q_{2}^{})_{I}^{} e_{q_{2}^{*} \omega }^{} q_{2}^{*} - (q_{2}^{})_{I}^{} e_{ q_{2}^{*} \omega }^{} q_{1}^{*}= 
 \dif \psi_{\Wund}^{} + \psi_{\Wund}^{} \dif .
$$
Moreover, it is obvious that
$ (q_{2}^{})_{I}^{} e_{q_{2}^{*} \omega }^{} q_{2}^{*} = \id_{\Omega_{c}^{}(\Wund)}^{}$,
so we obtain 
\begin{align}
p_{I}^{} e_{\omega}^{} p_{}^{*} - T_{}^{\prime}T
&= 
p_{I}^{} \parens{\id_{\Omega_{c}^{}(\Wund)}^{} -  (q_{2}^{})_{I}^{} e_{ q_{2}^{*} \omega }^{} q_{1}^{*} }
e_{\omega}^{} p_{}^{*}
\nonumber\\
&=
 p_{I}^{} \parens{ \dif \psi_{\Wund}^{} + \psi_{\Wund}^{} \dif } e_{\omega}^{} p_{}^{*}
\nonumber\\
&=
 \dif \circ p_{I}^{} \psi_{\Wund}^{} e_{\omega}^{} p_{}^{*} +
p_{I}^{} \psi_{\Wund}^{} e_{\omega}^{} p_{}^{*} \circ \dif.
\label{EqCochainHomotopy010}
\end{align}
Remark that $p_{I}^{}\circ \dif = \dif \circ p_{I}^{}$ 
because the act on differential forms with compact support,
and $e_{\omega}^{} \circ \dif = \dif \circ e_{\omega}^{}$ because $\omega$ is a closed form.

Next let us consider submersion $p_{X}^{} \colon BX \to X$ given by
$(x,v) \mapsto \exp_{x}^{}(v)$.
Note that $p = p_{X}^{} \widetilde{f}$.
\begin{align*}
\xymatrix{
f_{}^{*}BX \ar[rd]_{ p } \ar[r]^{\widetilde{f}}
&
BX \ar[d]^{ p_{X}^{} }
\\ 
&
X
}
\end{align*}
Now we want to check that $p_{I}^{} e_{\omega}^{} p_{}^{*} = (p_{X}^{})_{I}^{} e_{\omega_{0}^{}}^{} p_{X}^{*}$.
For any $\nu \in \Omega_{c}^{*}(X)$ and $\zeta \in \Omega_{c}^{*}(BX)$,
\begin{align*}
\int_{X} \nu \wedge p_{I}^{} \parens{\widetilde{f}_{}^{*} \zeta}
&=
\int_{\Wund} p_{}^{*} \nu \wedge \widetilde{f}_{}^{*} \zeta
=
\int_{\Wund} \widetilde{f}_{}^{*}\parens{ p_{X}^{*} \nu \wedge \zeta}
\\
&=
\mathrm{deg} \parens{ \widetilde{f} }
\int_{BX} p_{X}^{*} \nu \wedge \zeta
 =
\int_{BX} p_{X}^{*} \nu \wedge \zeta
 = 
\int_{X} \nu \wedge (p_{X}^{})_{I}^{} (\zeta),
\end{align*}
since $f$ is an orientation preserving proper homotopy equivalent. In particular,
we obtain
$$
p_{I}^{}\parens{\widetilde{f}_{}^{*} \zeta} = (p_{X}^{})_{I}^{} (\zeta).
$$
Put $\zeta := p_{X}^{*} \xi \wedge \omega_{0}^{}$ for $\xi \in \Omega_{c}^{*}(X)$
to obtain 
\begin{align}
p_{I}^{} e_{\omega}^{} p_{}^{*}(\xi)
=
p_{I}^{} \parens{\widetilde{f}_{}^{*} p_{X}^{*} \xi \wedge \widetilde{f}_{}^{*} \omega_{0}^{}}
& =
p_{I}^{} \parens{ \widetilde{f}_{}^{*} (p_{X}^{*} \xi \wedge \omega_{0}^{})}
\nonumber\\
& = (p_{X}^{})_{I}^{} (p_{X}^{*} \xi \wedge \omega_{0}^{})
= (p_{X}^{})_{I}^{} e_{\omega_{0}^{}}^{} p_{X}^{*} (\xi).
\label{EqCochainHomotopy020}
\end{align}
Let $\pi \colon BX \to X$ be the natural projection.
Since $p_{X}^{}(x,0) = \pi(x,0)$, by Proposition \ref{KeyPropChainHomotopy},
there exists
a properly supported $G$-equivariant $L_{}^{2}$-bounded operator 
$\psi_{X}^{} \colon \Omega_{c}^{*}(X) \to \Omega_{c}^{*}(X)$ satisfying
\begin{align}
\pi_{I}^{} e_{\omega_{0}^{} }^{} \pi_{}^{*} - (p_{X}^{})_{I}^{} e_{\omega_{0}^{}}^{} p_{X}^{*} = 
 \dif \psi_{X}^{} + \psi_{X}^{} \dif. 
\label{EqCochainHomotopy030}
\end{align}
On the other hand, it is obvious that
$ \pi_{I}^{} e_{\omega_{0}^{} }^{} \pi_{}^{*} = \id_{\Omega_{c}^{}(X)}^{}$.
Therefore, combining (\ref{EqCochainHomotopy010}),
 (\ref{EqCochainHomotopy020}) 
 and (\ref{EqCochainHomotopy030}),
we conclude
\begin{align*}
\mathrm{id}_{\Omega_{c}^{}(X)}^{} - T_{}^{\prime}T
= \dif \phi + \phi \dif,
\end{align*}
where $\phi = p_{I}^{} \psi_{\Wund}^{} e_{\omega}^{} p_{}^{*} + \psi_{X}^{}$.
Since $\phi$ is properly supported $G$-invariant $L_{}^{2}$-bounded operator, 
it defines an element in $\mathbb{L}(\E_{X}^{})$.
\end{proof}
\begin{proof}[Proof of Theorem \ref{MainThm}]
First, let us check that
$T$ satisfies the assumption (1) of Lemma \ref{LemHilSkPerturb}.
Since $\omega$ is a closed form and has fiber-wisely compact support,
it follows that $T \delta_{X}^{} = \delta_{Y}^{} T$.
Let $\map{g}{X}{Y}$ be the $G$-equivariant homotopy inverse of $f$
and consider a map $S \in \mathbb{L}(\E_{Y}^{}, \E_{X}^{})$ 
constructed in the same method as $T$ from $g$ instead of $f$
in Definition \ref{ConstructionOfTf}.
By \ref{LemCompositionOfT}, the composition $ST$ is equal to the map
$T_{p}^{} \in \mathbb{L}(\E_{X}^{})$ for $p$ satisfying that
$p(\cdot, 0)$ is $G$-equivariant homotopic to $\mathrm{id}_{X}^{}$.
Then by Proposition \ref{KeyPropChainHomotopy}, 
there exists $\phi_{X}^{} \in \mathbb{L}(\E_{X}^{})$ satisfying that
$ST - (\delta_{X}^{} \phi_{X}^{} + \phi_{X}^{}\delta_{X}^{}) = T_{\mathrm{id}_{X}^{}}^{} = \mathrm{id}_{\E_{X}^{}}^{} $.
Thus, $ST$ induces the identity map on $\Ker(\delta_{X}^{}) / \Im(\delta_{X}^{})$.
Similarly $TS$ induces the identity map on $\Ker(\delta_{Y}^{}) / \Im(\delta_{Y}^{})$,
and hence, $T$ induces an isomorphism 
$\Ker(\delta_{X}^{}) / \Im(\delta_{X}^{}) \to \Ker(\delta_{Y}^{}) / \Im(\delta_{Y}^{})$.

The assumption (2) of Lemma \ref{LemHilSkPerturb}
is obtained from \ref{PropAssumptionTwo}.

Finally, let $\ep(\xi) := (-1)^{k} \xi$ for $\xi \in \Omega_{c}^{k}(X)$.
Clearly, $\ep$ determines an operator $\ep \in \mathbb{L}(\E_{X}^{})$, $\ep^{2}=1$ and
satisfies $\ep^{\prime} = \ep$,
$\ep (\dom (\delta_{X}^{}) ) \subset \dom(\delta_{X}^{})$
and $\ep \delta_{X}^{} = -\delta_{X}^{} \ep$.
Moreover since neither $T$ nor $T_{}'$ changes the order of the differential forms,
$\ep$ commutes with $1-T_{}' T$.
Thus $\ep$ satisfies the assumption (3) of Lemma \ref{LemHilSkPerturb}.
To conclude, we obtain
$\mathrm{ind}_{G}^{}(\partial_{X}^{}) = \Psi (\E_{X}^{}, Q_{X}^{}, \delta_{X}^{}) =
\Psi (\E_{Y}^{}, Q_{Y}^{}, \delta_{Y}^{}) = \mathrm{ind}_{G}^{}(\partial_{Y}^{})$.
\end{proof}
\subsection{On proof of Corollary \ref{MainCor}}
To prove Corollary \ref{MainCor}, we will combine 
\cite[Theorem A]{Fu16} with Theorem \ref{MainThm}.
Suppose, in addition, that $G$ is unimodular and $H_{1}^{}(X; \R) = H_{1}^{}(Y; \R) = \{ 0 \}$.
Let $\map{f} {Y} {X}$ be a $G$-equivariant orientation preserving homotopy invariant map
and consider a $G$-manifold $Z := X \sqcup (-Y)$, 
the disjoint union of $X$ and orientation reversed $Y$.
Let $\partial_{Z}^{}$ be the signature operator, then
we have that 
$\ind_{G}^{}(\partial_{Z}^{}) = \ind_{G}^{}(\partial_{X}^{}) - \ind_{G}^{}(\partial_{Y}^{}) = 0 
\in K_{0}^{}(C^{*}(G))$.
Although the $G$-manifold should be connected in \cite[Theorem A]{Fu16},
however in this case, we can apply it to $Z$ after replacing some arguments in \cite{Fu16} as follows.

\cite[subsection 6.1 and 6.2]{Fu16}
When constructing a $U(1)$-valued cocycle $\alpha \in Z(G;U(1))$ 
from the given line bundle,
we just use a line bundle $L$ over $X$ ignoring $f_{}^{*}$ over $Y$.
When constructing family of line bundles $\{ L_{t}^{} \}$ on which the central extension group
$G_{\alpha^{t}}^{}$ acts,
just construct a family of line bundles $\{ L_{t}^{} \}$ over $X$ in the same way and pull back on $Y$
to obtain a family $\{ f_{}^{*}L_{t}^{} \}$.
To be specific, $f_{}^{*}L_{t}^{}$ is a trivial bundle $Y \times \C$,
equipped with the connection given by $\nabla^{t} = \dif +it f_{}^{*}\eta$ and
the action of $G_{\alpha^{t}}^{}$ is given by
$$
(\gamma, u)(y,z)=(\gamma y,\ \exp[-it f_{}^{*}\psi_{\gamma}(x)]uz)
\qqfor
(\gamma, u) \in G_{\alpha^t}^{}, \ 
y \in Y,\ 
z \in \C=(L_{t})_{x}.
$$
Then consider a family of  $G_{\alpha^{t}}^{}$-line bundles $\{ L_{t}^{} \sqcup f_{}^{*}L_{t}^{} \}$ over $Z$.
We also need the similar replacement in \cite[Definition 7.19]{Fu16} to obtain the global section
on $L_{t}^{} \sqcup f_{}^{*}L_{t}^{}$.
Then the rest parts proceed similarly.
\section{Index of Dirac operators twisted by almost flat bundles}
\label{SectionAlmFlat}
Now we will discuss on the Dirac operators twisted by a family of Hilbert module bundles
$\{ E_{}^{k}\}$ whose curvature tend to zero. 
and prove Theorem \ref{ThmAlmostFlatIndex}.
Such an family is called a family of almost flat bundles.
In this section, it is convenient to formulate the index map
using $K{\!}K$-theory.
\subsection{$G$-index map in $K{\!}K$-theory}
\begin{lem}
{\upshape{\cite[Theorem 3.11]{Kas88}}}
Let $G$ be a second countable locally compact Hausdorff group.
For any $G$-algebras $A$ and $B$ there exists
a natural homomorphism
$$
\map{j^{G} } {K{\!}K^{G}(A,B)} {K{\!}K(C^{*}(G;A), C^{*}(G;B))}
$$
Furthermore
if $x\in K{\!}K^{G}(A,B)$ and $y\in K{\!}K^{G}(B,D)$,
then $j^{G}(x\hattensor_{B}^{}y) = j^{G}(x) \hattensor_{C^{*}(G;B)}^{}
j^{G}(y)$.
\QED
\end{lem}
\begin{lem}
\label{CutoffKHomologyClass}
Using a cut-off function $c\in C_c(X)$,
one can define an idempotent $p \in C_c(G;C_0(X))$
by the formula;
$$\check{c}(\gm)(x) = \sqrt{c(x) c(\gminv x) \Delta(\gm)_{}^{-1}}.$$
In particular it defines an element of K-homology
denoted by $[c]\in K_0(C^{*}(G;C_0(X)))$.
Moreover the element of $K$-homology
$[c]\in K_0(C^{*}(G;C_0(X)))$ does not depend on
the choice of cut-off functions.
\QED
\end{lem}
\begin{dfn}[$G$-Index]\cite[Theorem 5.6.]{Kas16}
\label{BCmap}
Define
$$
\map{\mu_{G}}{K{\!}K^{G}(C_0(X), \C)}{K_{0}(C^*(G))}
$$
as the composition of
\begin{itemize}
\item
$\map{j^{G}}{K{\!}K^{G}(C_0(X), \C)}
{K{\!}K(C^*(G;C_0(X)) , C^*(G))}$
and
\item
$\map{[c]\hattensor}{K{\!}K(C^*(G;C_0(X)) , C^*(G))}
{K{\!}K(\C , C^*(G))} \simeq K_{0}(C^*(G))$,
i.e.,
\begin{eqnarray*}
\mu_{G}(\mathchar`-) :=
[c]\hattensor_{C^*(G;C_0(X))} j^{G}(\mathchar`-)
\quad \in K_{0}(C^*(G) ).
\end{eqnarray*}
\end{itemize}
\end{dfn}
\begin{rem}
\label{RemReductionOnDirac}
As in \cite[Remark 4.4.]{Kas16} or \cite[Subsection 5.2]{Fu16},
it is sufficient to consider only in the case of Dirac type operators
for calculating the index.
\end{rem}
Let $B$ be a unital $C^{*}$-algebra.
Following the definition \ref{BCmap},
we define the index maps with coefficients;
\begin{dfn}
For unital $C^*$-algebras $B$,
define the index map
$$
\map{\mathrm{ind}_{G}^{}}{K\!K^{G}(C_{0}^{}(X), B)}
{K_{0}(C^{*}(G;B))}
$$
as the composition of
\begin{itemize}
\item
$\map{j^{G}}{K\!K^{G}(C_{0}^{}(X), B)}
{K\!K(C^*(G;C_{0}^{}(X)) ,\: C^{*}(G;B))}$
and
\item
$\map{[c]\hattensor}{K\!K(C^*(G;C_{0}^{}(X)) ,\: C^{*}(G;B))}
{K_{0}(C^{*}(G;B))}$,
i.e.,
\begin{eqnarray*}
\mathrm{ind}_{G}^{}(\mathchar`-) :=
[c]\hattensor_{C^*(G;C_{0}^{}(X))} j^{G}(\mathchar`-)
\quad \in K_{0}(C^{*}(G;B)).
\end{eqnarray*}
\end{itemize}
The crossed product $C^{*}(G;B)$ is either maximal or reduced one.
In this paper, we assume that 
$G$ acts on $B$ trivially. Then
$C^{*}_{\Max}(G; B)$ and $C^{*}_{\red}(G; B)$
will be naturally identified with
$C^{*}_{\Max}(G) {\otimes_{\Max}^{}} B$ and
$C^{*}_{\red}(G) {\otimes_{\min}^{}} B$ respectively.
Moreover if $B$ is nuclear, ${\otimes_{\Max}^{}} B$ and ${\otimes_{\min}^{}} B$ are identified.
\end{dfn}
\begin{dfn}
Let $E$ be a finitely generated projective $(\Z/2\Z)$-graded Hilbert $B$-module $G$-bundle.
Define $C_{0}^{} \parens{X; E}$ as a space consisting of
sections $\map {s} {X} {E}$ vanishing at infinity.
It is considered as a $\Z/2\Z$-graded Hilbert $C_{0}^{}\parens{X; B}$-module
with the right action given by point-wise multiplications
and the scalar product given by
$$
\angles{ s_{1}^{}, s_{2}^{}} (x) := \angles{s_{1}^{}(x), s_{2}^{}(x)}_{E_{x}^{}}
\in C_{0}^{}\parens{X;B}.
$$
\end{dfn}
\begin{rem}
The $C^{*}$-algebra
$ C_{0}^{}\parens{X; B}$ consisting of $B$-valued function vanishing at infinity
is naturally identified with $C_{0}^{}(X)\hattensor B$
by \cite[6.4.17. Theorem]{We93}.
Similarly, if $E = X\times E_{0}^{}$ is a trivial Hilbert $B$-module bundle over $X$,
then $ C_{0}^{}\parens{X; E}$
is naturally identified with $C_{0}^{}(X)\hattensor E_{0}^{}$ as Hilbert
$\parens{ C_{0}^{}\parens{X; B} \cong C_{0}^{}(X)\hattensor B }$-modules.
\end{rem}
\begin{dfn}
$E$ define an element in $K\!K$-theory
\begin{eqnarray*}
[E] = \parens{ C_{0}^{} \parens{X; E},\ 0 }
\in K\!K^{G}\parens{C_{0}^{}(X) ,\ C_{0}^{}(X)\hattensor B}.
\end{eqnarray*}
The action of $C_{0}^{}(X)$ on $C_{0}^{} \parens{X; E}$ is the
point-wise multiplication.
\end{dfn}
\begin{dfn}
Let $E$ be a finitely generated Hilbert $B$-module bundle over $X$
equipped with a Hermitian connection $\nabla_{}^{E}$.
Let $R_{}^{E} \in C_{}^{\infty}\parens{
X;\ \End(E) \otimes \turnV{2}(T_{}^{*}(X)) }$
denote its curvature.
Then define its norm as follows:
First, define the point-wise norm as the operator norm given by
$$
\norm{R_{}^{E}}_{x}^{} := \sup \set{ \norm{R_{}^{E}(u \wedge v)}_{\mathrm{\mathbb{L}(E)}}^{} }
{u, v \in T_{x}^{}X,\ \norm{u \wedge v} = 1}
\qqfor x \in X.
$$
Then define the global norm as the supremum in $x \in X$ of the point-wise norm;
$\norm{R_{}^{E}} := {\displaystyle \sup_{x\in X}\norm{R_{}^{E}}_{x}^{}}$
\end{dfn}
To describe the Theorem which we will prove,
\setcounter{thmalph}{2}
\begin{thmalph}
\label{ThmAlmostFlatIndex}
Let $X$ be a complete oriented Riemannian manifold
and let $G$ be a locally compact Hausdorff group
acting on $X$ isometrically, properly and co-compactly.
Moreover we assume that $X$ is simply connected.
Let $D$ be a $G$-invariant properly supported elliptic operator of order $0$
on $G$-Hermitian vector bundle over $X$.

Then there exists $\ep > 0$ satisfying the following:
for any finitely generated projective
Hilbert $B$-module $G$-bundle $E$ over $X$
equipped with a $G$-invariant Hermitian connection
such that $\norm{R_{}^{E}} < \ep$,
we have 
$$
\mathrm{ind}_{G}^{}\parens{ [E] \hattensor_{C_{0}^{}(X)}^{} [D] }= 0 
\quad \in K_{0}^{} \parens{C_{\Max}^{*}(G) {\otimes_{\Max}^{}} B}
$$
if $\mathrm{ind}_{G}^{}([D]) =0 \in K_{0}^{}(C_{\Max}^{*}(G))$.
If we only consider commutative $C^{*}$-algebras for $B$, then the same conclusion
is also valid for $C_{\red}^{*}(G)$.
\end{thmalph}
\subsection{Infinite product of $C^{*}$-algebras}
\label{SectionAssembling}
\begin{dfn}
Let $B_{k}$ be a sequence of $C^{*}$-algebras.
\begin{itemize}
\item
Define $\prod\limits_{k\in \N} B_{k}$ 
as the $C^{*}$-algebra consisting of norm-bounded sequences
$$
\prod_{k\in \N} B_{k}:=
\set{
\braces{b_1, b_2, \ldots} }{
b_{k} \in B_{k},\
\sup_{k} \{ \norm{b_{k}}_{B_{k}} \} < \infty
}.
$$
The norm of $B_{k}$ is given by 
$\norm{\braces{b_1, b_2, \ldots}}_{\prod B_{k}}:=
\sup_{k} \{ \norm{b_{k}}_{B_{k}}  \}.$
\item
Let $\bigoplus\limits_{k\in \N} B_{k}$ a closed two-sided ideal  in $\prod_{k\in \N} B_{k}$
consisting of sequences vanishing at infinity
$$
\bigoplus_{k\in \N} B_{k}:=
\set{
\braces{b_1, b_2, \ldots} }{
b_{k} \in B_{k},\
\lim_{k \to \infty} \norm{b_{k}} = 0
}.
$$
In other words, $\bigoplus_{k\in \N} B_{k}$ is a closure of the 
sub-space in $\prod_{k\in \N} B_{k}$ consisting of sequences $\braces{ b_1, b_2 , \ldots, 0, 0, \ldots}$
whose entries are zero except for finitely many of them.
\item
Define $\calq_{k\in \N} B_{k}$ as the quotient algebra given by
$$
\calq_{k\in \N} B_{k} := 
\parens{ \prod B_{k} } \big/ \parens{ \bigoplus B_{k} }.
$$
The norm of $\calq B_{k}$ is given by 
$\norm{\braces{b_1, b_2, \ldots}}_{\calq B_{k}}:=
\displaystyle{\limsup_{k\to \infty}} \norm{b_{k}}_{B_{k}} .$
\item
If $\E_{k}$ are Hilbert $B_{k}$-modules, 
one can similarly define $\prod \E_{k}$ as a Hilbert $\prod B_{k}$-module 
consisting of bounded sequences 
$$
\prod_{k\in \N} \E_{k}:=
\set{
\braces{s_{1}^{}, s_{2}^{}, \ldots} }{
s_{k} \in \E_{k},\
\sup_{k} \{ \norm{s_{k}}_{\E_{k}} \} < \infty
}.
$$
The action of $\prod B_{k}$ and $\prod B_{k}$-valued scalar product are defined
as follows;
\begin{eqnarray*}
\{s_{k}\} \cdot \{b_{k}\} &:=& \{ s_{k} \cdot b_{k} \} \in \prod \E_{k}
\qqfor \{s_{k}\} \in \prod \E_{k},\  \{b_{k}\} \in \prod B_{k}, \\
\angles{ \{ s_{k}^{1}\} ,\  \{ s_{k}^{2}\} }_{\prod \E_{k}} 
&:=& \braces{ \angles{s_{k}^{1},\ s_{k}^{2}}_{\E_{k}} }
\in \prod B_{k}
\qqfor \{ s_{k}^{1} \} , \{ s_{k}^{2}\} \in \prod \E_{k}.
\end{eqnarray*}

One can define similraly
$$
\bigoplus_{k\in \N} \E_{k}:=
\set{
\braces{s_{1}^{}, s_{2}^{}, \ldots} }{
s_{k} \in \E_{k},\
\lim_{k \to \infty} \norm{s_{k}}_{\E_{k}} =0
}
$$
as a Hilbert $\prod B_{k}$-module, and define
$$
\calq_{k\in \N} \E_{k}
:= \parens{ \prod \E_{k}} \hattensor_{\pi} \parens{ \calq B_{k} }
=
\parens{ \prod \E_{k}} \big/ \parens{ \bigoplus \E_{k}} 
$$
as a Hilbert $\calq B_{k}$-module, 
where $\map {\pi} {\prod B_{k}} {\calq B_{k}}$
denotes the projection.
\end{itemize}
\end{dfn}
\begin{exa}
If all of $B_{k}$ are $\C$,
then, $\prod \C = \ell^{\infty}(\N)$ and $\bigoplus \C = C_{0}(\N)$.
\end{exa}

Following \cite[Section 3.]{Ha12},
we will construct
``infinite product bundle $\prod E_{k}^{}$''
over $X$ which has a structure of
finite generated projective $\prod B_{k}^{}$-module.

\newcommand{\path}{\mathcal{P}_1}
\begin{dfn}
Let us fix some notations about the holonomy.
\begin{itemize}
\item
Two paths $p_{0}^{}$ and $p_{1}^{}$ from $x$ to $y$ in $X$ are thin homotopic to each other if
there exists an end points preserving homotopy $\map{h} {[0,1]\times [0,1]} {X}$
with $h(\cdot, j) = p_{j}^{}$
that factors through a finite tree $T$,
$$
h \colon [0,1]\times [0,1] \to T \to X
$$
such that both restrictions of the first map
$[0,1]\times \{j \} \to T$ are piecewise-linear for $j = 0,1$.
\item
The path groupoid $\path(X)$ is a groupoid consisting of all the points
in $X$ as objects.
The morphism from $x$ to $y$ are the equivalence class of piece-wise smooth paths
connecting given two points
$$
\path(X)[x,y] :=
\set{\map{p}{[0,1]}{X}}{p(0)=x,\ p(1)=y}/\sim.
$$
The equivalent relationship is given by re-parametrization
and thin homotopy.
\item
If a Hilbert $B$-module $G$-bundle $E$ over $X$ is given,
the transport groupoid $\mathcal{T}(X;E)$ is a groupoid with the same objects as $\path(X)$.
The morphism from $x$ to $y$ are the unitary isomorphisms
between the fibers
$\mathcal{T}(X;E)[x,y] :=
\mathrm{Iso}_{B}^{}\parens{E_x, E_y}$.
\end{itemize}
\end{dfn}
\newcommand{\area}{\mathrm{area}}
\begin{dfn}
A parallel transport of $E$ is a continuous functor
$\Phi^{E} \colon \path(X) \to \mathcal{T}(X;E)$.
$\Phi^{E}$
is called $\ep$-close to the identity
if for each $x\in X$ and
contractible loop $p\in \path(X)[x,x]$,
it follows that
$$\norm{\Phi^{E}_{p} - \mathrm{id}_{E_x}} <\ep \cdot \area(D)$$
for any
two dimensional disk $D\subset X$ spanning $p$.
$D$ may be degenerated partially or completely.
\end{dfn}
\begin{rem}
Let $E$ be a Hermitian vector bundle,
in other words, a finitely generated Hilbert $\C$-module bundle,
equipped with a compatible connection $\nabla$.
Let $\Phi^{E}$ be the parallel transport with respect to $\nabla$
in the usual sense.
 If
its curvature $R_{}^{E} \in C_{}^{\infty}\parens{
X;\  \End(E) \otimes \turnV{2}(T_{}^{*}(X)) }$
has uniformly bounded
operator norm $\norm{R_{}^{E}} < C$,
 then
for any loop $p\in \path(X)[x,x]$
and any two dimensional disk $D\subset X$ spanning $p$,
it follows that
$\norm{\Phi^{E}_{p} - \mathrm{id}_{E_x}} <\int_{D}\norm{R_{}^{E}}
<C \cdot \area(D)$
so it is $C$-closed to identity.
\end{rem}
\begin{prop}
\label{ConstructProductBundle}
Let $\{E^{k}\}$ be a sequence of Hilbert $B_{k}$-module
$G$-bundles over $X$
with $B_{k}$ unital $C^{*}$-algebras.
Assume that each parallel transport
$\Phi^{k}$ for $E^{k}$ is $\ep$-close to the identity
uniformly, that is, $\ep$
is independent of $k$.

Then there exists a finitely generated Hilbert
$\parens{\prod_{k} B_{k}}$-module
$G$-bundle $V$ over $X$
with Lipschitz continuous transition functions in diagonal form
and so that the $k$-th component of this bundle
is isomorphic to the original $E^{k}$.

Moreover,
if the parallel transport
$\Phi^{k}$ for each of $E^{k}$ comes from the
$G$-invariant connection $\nabla^{k}$ on $E^{k}$,
$V$ is equipped with a continuous $G$-invariant
connection induced by $E^{k}$.
\end{prop}
\begin{proof}
We will essentially follow the proof of  \cite[Proposition 3.12.]{Ha12}.
For each $x\in X$ take a open ball $U_{x} \subset X$
of radius $\ll 1$ whose center is $x$.
Assume that each $U_{x}$ is geodesically convex.
Due to the corollary \ref{CorSliceThm} of the slice theorem,
there exists a sub-family of
finitely many open subsets $\{U_{x_1}, \ldots U_{x_N} \}$
such that
$X=
\bigcup_{\gm \in G} \bigcup_{i=1}^{N} \gm (U_{x_i})
$.

Fix $k$.
In order to simplify the notation,
let $U_i := U_{x_i}$
and $\map{\Phi_{y; x}}{E^{k}_{y}}{E^{k}_{x}}$ denote
the parallel transport of $E^{k}$
along the minimal geodesic from $y$ to $x$
for $x$ and $y$
in the same neighborhood $\gm(U_{i})$.
Trivialize $E^{k}$ via
$\map{\Phi_{y;x_i}}{E^{k}_{y}}{E^{k}_{x_i}}$
on each $U_{i}$.
Similarly  trivialize $E^{k}$ on each $\gm(U_{i})$ for $\gm \in G$
via $\map{\Phi_{\gm y; \gm x_i}} {E^{k}_{\gm y}} {E^{k}_{\gm x_i}}$.
Note that since parallel transport commute with the action of $G$,
it follows that
$\Phi_{\gm y, \gm x_i} = \gm \circ \Phi_{y; x_i} \circ \gm_{}^{-1}$.

These provide a local trivializations for $E^{k}$ whose
transition functions have uniformly bounded Lipschitz constants.
More precisely
we have to fix unitary isomorphisms
$\map{\phi_{\gm x_i}}{E^{k}_{\gm x_i}}{\mathcal{E}^{k}}$
between the fiber on $\gm x_i$ and
the typical fiber $\mathcal{E}^{k}$.
Our local trivialization is
$
\map{\phi_{\gm x_i} \Phi_{y;\gm x_i}}{E^{k}_{y}}{\mathcal{E}^{k}}
$.
If $y,z\in \gm(U_{i}) \cap \eta(U_{j}) \neq \emptyset $,
we can consider the transition function
$$
y\mapsto
\psi_{\gm(U_{i}),\eta(U_{j})}^{}(y):=
\parens{\phi_{\eta x_j} \circ \Phi_{y;\eta x_j}}
\parens{\phi_{\gm x_i}\circ \Phi_{y;\gm x_i}}^{-1}
\in \End_{B_{k}}(\mathcal{E}^{k}).
$$
Now we will estimate its Lipschitz constant as follows;
\newcommand{\AAAA}{\Phi_{y;\eta x_j}^{}}
\newcommand{\BBBB}{\Phi_{y;\gm x_i}^{-1}}
\newcommand{\CCCC}{\Phi_{z;\eta x_j}^{}\Phi_{y;z}^{}}
\newcommand{\DDDD}{\Phi_{y;z}^{-1}\Phi_{z;\gm x_i}^{-1}}
\begin{eqnarray*}
&&
\psi_{\gm(U_{i}),\eta(U_{j})}^{}(y) - \psi_{\gm(U_{i}),\eta(U_{j})}^{}(z)
\\
&=&
 \parens{\phi_{\eta x_j} \Phi_{y;\eta x_j}}\parens{\phi_{\gm x_i} \Phi_{y;\gm x_i}}^{-1}
-
 \parens{\phi_{\eta x_j} \Phi_{z;\eta x_j}}\parens{\phi_{\gm x_i} \Phi_{z;\gm x_i}}^{-1}
\\
&=&
\phi_{\eta x_j}
\braces{
\parens{\AAAA} \parens{\BBBB}
-
\parens{\CCCC} \parens{\DDDD}
}
\phi_{\gm x_i}^{-1}
\\
&=&
\phi_{\eta x_j}
\braces{
\parens{\AAAA - \CCCC} \parens{\BBBB}
+
\parens{\CCCC} \parens{\BBBB - \DDDD}
}
\phi_{\gm x_i}^{-1}.
\end{eqnarray*}
Since $\phi$'s and $\Phi$'s are isometry, it follows that
\begin{eqnarray}
\norm{\psi_{\gm(U_{i}),\eta(U_{j})}^{}(y) - \psi_{\gm(U_{i}),\eta(U_{j})}^{}(z)}
&\leq&
\norm{\AAAA - \CCCC}
+
\norm{\BBBB - \DDDD}
\nonumber\\
&=&
\norm{\Phi_{y; \eta x_j} \Phi_{z;y}\Phi_{\eta x_j; z} - \id_{E^{k}_{\eta x_j}}}
+
\norm{\Phi_{z; \gm x_i} \Phi_{y;z}\Phi_{\gm x_i; y} - \id_{E^{k}_{\gm x_i}}}
\nonumber\\
&\leq&
\ep \cdot (\area (D_1) + \area (D_2)).
\label{BoundedLipConst}
\end{eqnarray}
Here $D_1 \subset \eta(U_{j})$ is a two dimensional disk spanning
the piece-wise geodesic loops connecting
$\eta x_j$, $y$, $z$, and $\eta x_j$
and $D_2 \subset \gm(U_{i})$ is a two dimensional disk spanning
the piece-wise geodesic loop connecting
$\gm x_i$, $y$, $z$, and $\gm x_i$.

We claim that there exists a constant $C$ depending
only on $X$ such that
\begin{eqnarray}
\label{AreaOfD1D2}
\area (D_1), \: \area(D_2)
\leq
C\cdot \mathrm{dist}(y,z)
\end{eqnarray}
if we choose suitable disks $D_1$ and $D_2$.

We verify this
using the geodesic coordinate
$\exp_{\eta x_j}^{-1} \colon \eta(U_{j}) \to T_{\eta x_j}X$
centered at $\eta x_j \mapsto 0$.
More precisely,
let $p$ denote the minimal geodesic from
$y=p(0)$ to $z=p(\mathrm{dist}(y,z))$ with unit speed.
 Consider
$$
D_0:=\braces{(r \cos \theta, r \sin \theta)\in \R^{2} \;\big|\;
0\leq r \>,\; 0\leq \theta \leq \mathrm{dist}(y,z) }
\subset \R^{2}
$$
and $\map{F}{D_0}{\eta(U_j) \subset X}$ given by
\begin{eqnarray*}
F(r \cos \theta, r \sin \theta)
:=
\exp_{\:\eta x_j}\parens{r \exp_{\eta x_j}^{-1}(p(\theta))}.
\end{eqnarray*}
Set $D_1 := F(D_0)$.
$F$ is injective if $\exp_{\eta x_j}^{-1}(y)$ and $\pm\exp_{\eta x_j}^{-1}(z)$
are on different radial directions, in which case
$F$ is a homeomorphism onto its image, and hence
$F(D_0)$ is a two dimensional disk spanning the target loop.
The Lipschitz constant of $F$ is bounded by a constant
depending on the curvature on $\eta(U_{j})$,
so there exists a constant $C_{\eta ,j}$ depending on the
Riemannian curvature on $\eta(U_j)$ satisfying
$$
\area(D_1) \leq C_{\eta ,j} \cdot \area(D_0)
\leq C_{\eta ,j} \cdot \mathrm{dist}(y,z).
$$
However, the constant $C_{\eta ,j}$ can be taken independent of
$\eta(U_j)$ due to the bounded geometry of $X$ implied by the slice theorem (Corollary
\ref{CorSliceThm}).
In the case of
$\exp_{\eta x_j}^{-1}(y)$ and $\pm\exp_{\eta x_j}^{-1}(z)$
are on the same radial direction, $D_1$
is completely degenerated and $\area(D_1) = 0$.
We can construct $D_2$ in the same manner
so the claim (\ref{AreaOfD1D2}) has been verified.

Therefore combining (\ref{BoundedLipConst})
and (\ref{AreaOfD1D2}), we conclude that
the Lipschitz constants of the transition functions
of these local trivialization
are less than
$2C\ep$,
which are independent of $E^{k}$, $U_{i}$ and $\gm \in G$,
in particular, the product of them
$$
\Psi_{\gm(U_{i}),\eta(U_{j})} :=
\braces{\psi^{k}_{\gm(U_{i}),\eta(U_{j})}}_{k\in \N}
\colon
\gm(U_{i}) \cap \eta(U_{j})
\to
\mathbb{L}_{\parens{\prod B_{k}}}\parens{\prod_{k} \mathcal{E}^{k}}
$$
are Lipschitz continuous.
So it is allowed to use them to define the
Hilbert $\prod_k B_{k}$-module bundle $V$ as required.
 Precisely
$V$ can be constructed as follows;
$$
V := \bigsqcup_{\gm ,i} \parens{ \gm(U_{i}) \times
\prod_{k} \mathcal{E}^{k}}
\Big/ \sim.
$$
Here, $(x,v) \in \gm(U_{i}) \times
\prod_{k} \mathcal{E}^{k}$ and
$(y,w) \in \eta(U_{j}) \times
\prod_{k} \mathcal{E}^{k}$
are equivalent if and only if
$x=y \in \gm(U_{i}) \cap \eta(U_{j}) \neq \emptyset$ and
$\Psi_{\gm(U_{i}),\eta(U_{j})}(v)=w$.
By the construction of $V$,
if $\map{p_{n}^{}}{\prod_k B_{k}}{B_{n}}$ denotes the projection
onto the $n$-th component,
$V \hattensor_{p_{n}^{}}^{} B_{n}$ is isomorphic to the original
$n$-th component $E^{n}$.

In order to verify the continuity of the induced connection,
let $\{ \mathbf{e}_{i}^{} \}$ be any orthonormal local frame on
$U_{i}$ for an arbitrarily fixed $E^{k}$
obtained by the parallel transport along the minimal
geodesic from the center $x_i \in U_{i}$,
namely,
$\mathbf{e}_{i}^{}(y) = \Phi_{x_i ; y}\mathbf{e}_{i}^{}(x_i)$.
It is sufficient to verify that
$
\norm{\nabla^{k}\mathbf{e}_{i}^{}} < C
$.
Let $ v \in T_{y}X $ be a unit tangent vector and
$p(t) := \exp_{y}(tv)$ be the geodesic of unit speed
with direction $v$.
\begin{eqnarray*}
\nabla^{k}_{v}\mathbf{e}_{i}^{}(y)
&=&
\lim_{t\to 0}\frac{1}{t}\parens{
\Phi_{p(t);p(0)}\mathbf{e}_{i}^{}(p(t)) - \mathbf{e}_{i}^{}(p(0))
}
\\
&=&
\lim_{t\to 0}\frac{1}{t}\parens{
\Phi_{p(t);p(0)}\Phi_{x_i ; p(t)}
 - \Phi_{x_i ; p(0)}
}
\mathbf{e}_{i}^{}(x_i),
\\
\norm{\nabla^{k}_{v}\mathbf{e}_{i}^{}(y)}
&\leq&
\lim_{t\to 0}\frac{1}{|t|}\norm{
\Phi_{p(t);p(0)}\Phi_{x_i ; p(t)}
 - \Phi_{x_i ; p(0)}
}
\\
&\leq&
\lim_{t\to 0}\frac{1}{|t|}
\ep\cdot \area(D(t)),
\end{eqnarray*}
where $D(t)$ is a $2$-dimensional disk in $U_{i}$
spanning the piece-wise geodesic
connecting $x_{i}$, $p(0)=y$, $p(t)$ and $x_{i}$.
As above, we can find a constant $C>0$ and
disks $D(t)$ satisfying
$$
\area(D(t)) \leq C \cdot \mathrm{dist} (p(0), p(t))
= C |t|
$$
 for $|t| \ll 1$.
Hence, we obtain
$
\norm{\nabla^{k}_{v}\mathbf{e}_{i}^{}(y)}
\leq
C\ep.
$
 \end{proof}

\begin{dfn}
\label{ConstructW}
Let us define a Hilbert $\parens{\calq B_{k}^{} }$-module bundle
$$
W:= V\hattensor_{\pi}^{} \parens{\calq B_{k}^{} },
$$
where $\pi \colon \prod B_{k}^{} \twoheadrightarrow \calq B_{k}^{}$
denotes the projection.
\end{dfn}
The family of parallel transport
of $E^{k}$ induces the parallel transport
$\Phi^{W}$ of
$W$ which commutes with the action of $G$.
\begin{prop}
\label{FlatnessOfW}
If the parallel transport of $E^{k}$ is $C_k$-close
to the identity with $C_k \searrow 0$,
 then the $G$-bundle
$W$ constructed above is a flat bundle.
More precisely the parallel transport
$\Phi^{W}(p) \in \Hom(W_x, W_y)$ depends only on
the ends-fixing homotopy class of $p \in \path(X)[x,y]$.
\end{prop}
\begin{proof}
It is sufficient to prove that for any contractive loop
$p\in \path(X)[x,x]$,
it satisfies
$\Phi^{W}(p) = \mathrm{id}_{W_x}$.
Fix a two dimensional disk $D\subset X$ spanning the loop $p$.
For arbitrary $\ep >0$ there exists $n_0$ such that
every $k \geq n_0$ satisfies that
$\Phi^{E^{k}}$ is $\frac{\ep}{1+\area(D)}$-close to the identity.
\begin{eqnarray*}
\norm{\Phi^{W}(p) - \mathrm{id}_{W_x}^{}}
=
\limsup_{k\to \infty} \norm{\Phi^{E^{k}}(p) - \mathrm{id}}
\leq
\sup_{k\geq n_0} \norm{\Phi^{E^{k}}(p) - \mathrm{id}}
\leq
\frac{\ep}{1+\area(D)}\cdot\area(D)
\leq
\ep
\end{eqnarray*}
This implies $\Phi^{W}(p) = \mathrm{id}_{W_x}^{}$.
 \end{proof}
\subsection{Index of the product bundle}
\begin{prop}
\label{PiMu}
\text{}
\begin{enumerate}
\item
Let $\map{p_{n}^{}}{\prod B_{k}^{}}{B_{n}^{}}$ denote the projection
onto the $n$-th component and consider
$$\map{(1 \otimes p_{n}^{})_{*}^{}}
{K_{0} \parens{C^*(G) \hattensor \parens{\prod  B_{k}^{}}}}
{K_{0} \parens{C^*(G)\hattensor B_{n}^{}}}.
$$
 Then
$$
(1 \otimes p_{n}^{})_{*}^{}
\mathrm{ind}_{G}^{} \parens{ \brackets{\prod E_{}^{k}} \hattensor [D]}
=
\mathrm{ind}_{G}^{} \parens{ \brackets{E_{}^{n}} \hattensor [D]}.
$$
\item
Let $\map{\pi}{\prod B_{k}^{}}{\calq B_{k}^{}}$ denote the quotient map
and consider
$$\map{(1 \otimes \pi)_{*}^{}}
{K_{0} \parens{C^*(G) \hattensor \parens{\prod B_{k}^{}}}}
{K_{0} \parens{C^*(G)\hattensor \parens{\calq B_{k}^{}}}}.
$$
 Then
$$
(1 \otimes \pi)_{*}^{}
\mathrm{ind}_{G}^{} \parens{ \brackets{\prod E_{}^{k}} \hattensor [D]}
=
\mathrm{ind}_{G}^{} ([W] \hattensor [D]).
$$
\end{enumerate}
\end{prop}
\begin{proof}
As for the first part,
$\brackets{E_{}^{n}} = (p_{n}^{})_* \brackets{\prod E_{}^{k}}
\in K\!K^{G}(C_{0}^{}(X),\ C_{0}^{}(X)\hattensor B_{n}^{})$ by the construction of
$\prod E_{}^{k}$.
Then it follows that
\begin{eqnarray*}
\mathrm{ind}_{G}^{} \parens{ \brackets{E_{}^{n}} \hattensor [D]}
&=&
\mathrm{ind}_{G}^{} \parens{ (p_{n}^{})_* \brackets{\prod E_{}^{k}} \hattensor [D]}
\\
&=&
[c]\hattensor j^{G}
 \parens{\brackets{\prod E_{}^{k}} \hattensor [D] \hattensor{\mathbf{p_{n}^{}}}}
\\
&=&
[c]\hattensor j^{G} \parens{\brackets{\prod E_{}^{k}} \hattensor [D]}
\hattensor j^{G}({\mathbf{p_{n}^{}}})
\\
&=&
\mathrm{ind}_{G}^{} \parens{\brackets{\prod E_{}^{k}} \hattensor [D]}
\hattensor j^{G}({\mathbf{p_{n}^{}}})
\\
&=&
(1 \otimes p_{n}^{})_* \mathrm{ind}_{G}^{}
\parens{ \brackets{\prod E_{}^{k}} \hattensor [D]},
\end{eqnarray*}
where $\mathbf{p_{n}^{}} = (B_{n}^{}, p_{n}^{}, 0) \in K\!K \parens{ \prod_{k\in\N} B_{k}^{},\ B_{n}^{} }$.
Then note that
$
j^{G}(\mathbf{p_{n}^{}})
=
(C^{*}(G)\hattensor B_{n}^{}, 1 \hattensor p_{n}^{}, 0)
\in K\!K\parens{C^{*}(G) \hattensor \parens{ \prod_{k\in\N} B_{k}^{}},\ C^{*}(G) \hattensor B_{n}^{}}
$.
 Since
$\pi_{*}\brackets{\prod E_{}^{k}}
= [W] \in K\!K^{G}\parens{ C_{0}^{}(X), C_{0}^{}(X) \hattensor \parens{ \calq B_{k}^{}} }$
by the construction of $W$,
the second part can be proved in the similar way.
 \end{proof}
\begin{prop}\label{MuW0}
Let $[D]$ be a $K$-homology element in $K\!K^{G}(C_{0}^{}(X), \C)$
determined by a Dirac operator on a $G$-Hermitian vector bundle $\bb{V}$ over $X$.
Suppose that $W$ is a finitely generated flat $B$-module $G$-bundle.
Assume that $X$ is simply connected.

 Then
$\mathrm{ind}_{G}^{}\parens{[W]\hattensor [D]}=0
\in K_{0} \parens{ C^{*}(G) \hattensor B}$
if $\mathrm{ind}_{G}^{}([D]) = 0$.
\end{prop}
In order to prove this,
we introduce an element $[W]_{\mathrm{rpn}}\in K\!K^{G}(\C,B)$
using the holonomy representation.
\begin{dfn}\text{}
\begin{itemize}
\item
Let $\Phi_{x;y}$ denote the parallel
transport of $W$ along an arbitrary path
from $x\in X$ to $y\in X$.
Since $X$ is simply connected and $W$ is flat,
it depends only on the ends of the path.
\item
Let us fix a base point $x_0 \in X$
and $W_{x_0}$ be the fiber on $x_0$.
Define $[W]_{\mathrm{rpn}}$ as
$$[W]_{\mathrm{rpn}} := (W_{x_0}, 0) \in K\!K^{G}(\C,B)$$
The action of $G$ on $W_{x_0}$ is given by the
holonomy $\map{\rho}{G}{\End_{Q}(W_{x_0})}$
$$
\rho [\gm](w) = (\Phi_{x_0;\gm x_0})^{-1} \gm(w)
\qqfor \gm \in G, \: w \in W_{x_0}
$$
\end{itemize}
\end{dfn}
\begin{lem}
\label{LemIndTwistedByW}
$$
[W]\hattensor_{C_{0}^{}(X)}^{} [D] = [D]
\hattensor_{\C}^{} [W]_{\mathrm{rpn}}
\in K\!K^{G}(C_{0}^{}(X), B).
$$
\end{lem}
\begin{proof}
Recall that $[D]\in K\!K^{G}(C_{0}^{}(X), \C)$ is given by
$\left( L^2(X;\mathbb{V} ), F_{D}^{} \right) $,
where $F_{D}^{}$ denotes the operator $\frac{D}{\sqrt{1+D_{}^{2}}}$,
and that
$$
[W]\hattensor_{C_{0}^{}(X)}^{} [D] = \parens{
C_{0}^{}(X;W) \hattensor_{C_{0}^{}(X)}^{} L^2 (X;\mathbb{V})
,\ F_{D_{}^{W}}^{}
},
$$
where $D_{}^{W}$ is the Dirac operator twisted by $W$
acting on $L^{2}(X; W {\otimes} \bb{V})
\simeq C_{0}^{}(X;W) \hattensor_{C_{0}^{}(X)}^{}L^{2}(X;\bb{V})$,
that is,
$$
D_{}^{W} = 
\sum_{j} \parens{ \mathrm{id}_{W}^{} \otimes c(e_j) }
\parens{\nabla_{e_j}^{W}\otimes \mathrm{id}_{\mathbb{V}}^{} +
\mathrm{id}_{W}^{}\otimes \nabla_{e_j}^{\mathbb{V}}
},
$$
where $\{ e_j \}$ denotes a orthogonal basis for $TX$
and $c(\cdot)$ denotes the Clifford multiplication by $\mathrm{Cliff}(TX)$ on $\mathbb{V}$.
The action of $C_{0}^{} (X)$ on $C_{0}^{}(X; W)$ and
$L^2 (X; \mathbb{V})$ are
the point-wise multiplications.
 On the other hand, 
$$
[D] \hattensor_{\C}^{}[W]_{\mathrm{rpn}}
=
\left( L^2 (X;\mathbb{V}) \hattensor_{\C}^{}
W_{x_0},\ F_{D}^{} \hattensor 1 \right).
$$
The action of $C_{0}^{} (X)$ is the point-wise multiplications.
Note that the action of $G$
on $W_{x_0}$ is given by the holonomy representation $\rho$.
It is sufficient to give a $G$-equivariant isomorphism
$$
\map{\varphi}
{L^2 (X;\mathbb{V}) \hattensor_{\C}^{} W_{x_0}}
{C_{0}^{}(X;W) \hattensor_{C_{0}^{}(X)}^{} L^2 (X;\mathbb{V})},
$$
which is compatible
with $D_{W}$ and $D \hattensor 1$.
 Set a section for $W$ given by
\begin{align}
\overline{w} \colon x \mapsto \Phi_{x_0 ; x}^{}w \in W_x
\label{EqGlobalSectionOnW}
\end{align}
and define
$\varphi$ on a dense sub space $C_{c}^{}(X;\bb{V}) \hattensor W_{x_0}$ as
$$
\varphi (s\otimes w) := \overline{w} {\cdot} \chi \otimes s
\qquad \text{
for $s\in C_{c}^{} (X;\mathbb{V})$ and $w \in W_{x_0}$},
$$
where $\chi \in C_{0}^{}(X)$ is an
arbitrary compactly supported function on $X$
with values in $[0,1]$
satisfying that $\chi(x)=1$ for all $x\in \supp(s)$.

$\varphi$ is independent of the choice of $\chi$ and hence
well-defined. Indeed,
Let $\chi' \in C_{c}^{}(X)$ be another such function,
and let $\rho \in C_{c}^{}(X)$ be a compactly supported function on $X$
with values in $[0,1]$
satisfying that $\rho(x)=1$ for all
$x\in \supp(\chi) \cup \supp(\chi')$. Then
in $ C_{0}^{}(X; W) \hattensor_{C_{0}^{}(X)}^{} C_{c}^{}(X; \mathbb{V})$,
$$
\overline{w} {\cdot}  \chi \otimes s
 -  \overline{w} {\cdot} \chi' \otimes s
=
\overline{w} {\cdot}\parens{\chi-\chi'} \otimes s
=
\overline{w}  {\cdot} \rho {\cdot} \parens{\chi-\chi'} \otimes s
=
\overline{w}  {\cdot} \rho \otimes \parens{\chi-\chi'} s
=
 0.
$$

Now we obtain that
$$
D_{W} \circ \varphi (s\otimes w)
=
D_{W} (\overline{w} \otimes s)
=
\overline{w} \otimes D(s)
=
\varphi\circ (D\hattensor 1) (s\otimes w)
$$
for $s\in C_{c}^{} (\mathbb{V})$ and $w \in W_{x_0}$.
This is because $\nabla^{W} \overline{w} = 0$ by its construction.

Compatibility with the action of $G$ is verified as follows;
\begin{eqnarray*}
\varphi (\gm (s\otimes w))(x)
= \Phi_{x_0 ; x}^{}(\rho[\gm](w)) \otimes \gm(s(\gm_{}^{-1}x))
 &=& \Phi_{x_0 ; x}^{} (\Phi_{x_0;\gm x_0})^{-1} \gm(w) \otimes \gm(s(\gm_{}^{-1}x))
 \\
 &=& \Phi_{\gm x_0 ; x}(\gm(b)) \otimes \gm(s(\gm_{}^{-1}x)),
\\
\gm (\varphi (s\otimes w))(x)
 = \gm((\Phi_{x_0;\gm_{}^{-1}x})(w) \otimes s(\gm_{}^{-1}x))
 &=& \Phi_{\gm x_0; x}(\gm(w)) \otimes \gm(s(\gm_{}^{-1}x)).
\end{eqnarray*}

Let us check that $\varphi$ induces an isomorphism.
For $s_{1}^{}{\otimes} w_{1}^{},\ s_{2}^{}{\otimes} w_{2}^{} \in
C_{c}^{} (X; \mathbb{V}) \hattensor_{\C}^{} W_{x_0}
$, it follows that
\begin{eqnarray*}
&&
\Big\langle
\varphi(s_{1}^{} {\otimes} w_{1}^{}),\;
\varphi(s_{2}^{} {\otimes} w_{2}^{})
\Big\rangle_{C_{0}^{}(X; W) \hattensor_{C_{0}^{}(X; X)}^{} L^2 (X; \mathbb{V})}^{}
\\
&=&
\Big\langle
s_{1}^{},\;
\angles{\overline{w_{1}^{}}{\cdot} \chi,
 \overline{w_{2}^{}}{\cdot} \chi}_{C_{0}^{}(X; W)} s_{2}^{}
\Big\rangle_{L^2 (X; \mathbb{V})}^{}
\\
&=&
\int_{X} \Big\langle s_{1}^{}(x),\;
\angles{(\Phi_{x_0 ; x}^{}w_{1}^{})\chi(x),
(\Phi_{x_0 ; x}^{}w_{2}^{})\chi(x) }_{W_{x}}^{} s_{2}^{}(x)\Big\rangle_{\bb{V}_{x}}^{}
\:\dif \mathrm{vol}(x)
\\
&=&
\int_{X} \angles{w_{1}^{},w_{2}^{}}_{W_{0}}^{} \chi(x)^2 \angles{s_{1}^{}(x),
s_{2}^{}(x)}_{\bb{V}_{x}}^{}
\:\dif \mathrm{vol}(x)
\\
&=&
\angles{w_{1}^{},w_{2}^{}}_{W_{0}}^{} \angles{s_{1}^{},s_{2}^{}}_{L^{2}(X; \bb{V})}^{}
\\
&=&
\angles{
s_{1}^{} {\otimes} w_{1}^{},\;
s_{2}^{} {\otimes} w_{2}^{}
}_{L^{2}(X; \bb{V})\hattensor W_{x_0}}^{},
\end{eqnarray*}
where $\chi \in C_{0}^{}(X)$ is a compactly supported function on $X$
satisfying that $\chi(x)=1$ for all $x\in \supp(s_{1}^{}) \cup \supp(s_{2}^{})$.
This implies that $\varphi$ is continuous and injective.

Moreover, choose arbitrary $F \in C_{c}^{}(X; W)$ and
$s \in C_{c}^{}(X; \bb{V})$.
Since $\Phi_{x_0 ; x}^{-1}$ provides a trivialization of
$W \simeq X \times W_{x_{0}}^{}$, we have an isomorphism
$C_{c}^{}(X; W) \simeq C_{c}^{}(X) \hattensor_{\C}^{} W_{x_{0}}^{}$.
Remark that, however, this is not a $G$-equivariant isomorphism,
just as pre-Hilbert 
$\parens{ C_{0}^{}(X;B) \cong C_{0}^{}(X)\hattensor B }$-modules.
 Then there exist countable subsets $\{ f_{1}^{}, f_{2}^{}, \ldots \} \subset C_{c}^{}(X)$ and
$\{ w_{1}^{}, w_{2}^{}, \ldots \} \subset W_{x_0}^{}$
satisfying that $\sum_{j\in \N} f_{j}^{} \overline{ w_{j}^{} }= F$ in $C_{0}^{}(X; W) $.
Now it follows that
\begin{eqnarray*}
\varphi\parens{
\sum_{j\in \N} f_{j}^{} s \otimes w_{j}^{}
}
=
\sum_{j\in \N} \parens{ \overline{w_{j}^{}} {\cdot} \chi \otimes f_{j}^{} s }
=
\sum_{j\in \N} \parens{ \overline{w_{j}^{}}{\cdot} \chi f_{j}^{} \otimes s }
=
\parens{ \sum_{j\in \N} \overline{w_{j}^{}} f_{j}^{} }{\cdot} \chi \otimes s
=
F \otimes s,
\end{eqnarray*}
where $\chi \in C_{0}^{}(X)$ is a compactly supported function on $X$
satisfying that $\chi(x)=1$ for all $x \in \supp(F) \cup \supp(s)$.
This implies that the image of $\varphi$ is dense in
$C_{0}^{}(X;W) \hattensor L^{2}(X;\bb{V})$.
Therefore $\varphi$ induces an isomorphism.
 \end{proof}
\begin{proof}
[Proof of the Proposition \ref{MuW0}]
Due to the previous lemma, 
it follows that
\begin{eqnarray*}
\mathrm{ind}_{G}^{}\parens{[W]\hattensor [D]}
&=&
\mathrm{ind}_{G}^{} ([D] \hattensor [W]_{\mathrm{rpn}})
\\
&=&
[c] \hattensor j^{G}\parens{
[D] \hattensor [W]_{\mathrm{rpn}}
}
\\
&=&
[c] \hattensor \left(j^{G}[D]\right) \hattensor
\parens{j^{G} [W]_{\mathrm{rpn}} }
\\
&=&
(\mathrm{ind}_{G}^{}[D]) \hattensor \parens{j^{G} [W]_{\mathrm{rpn}} }.
\end{eqnarray*}
Thus the assumption $\mathrm{ind}_{G}^{} [D] =0$ implies
$\mathrm{ind}_{G}^{}\parens{[W]\hattensor [D]}=0$
 \end{proof}

\subsection{Proof of Theorem \ref{ThmAlmostFlatIndex}}
\newcommand{\ZZZZ}
{\mathrm{ind}_{G}^{} \parens{\brackets{\prod E_{}^{k}}\hattensor [D]}}
\newcommand{\WWWW}
{\mathrm{ind}_{G}^{} \parens{[W]\hattensor[D]}}
\newcommand{\XXXX}[1]
{\mathrm{ind}_{G}^{} \parens{ [E_{}^{#1}] \hattensor [D]}}
\begin{proof}
[Proof of Theorem \ref{ThmAlmostFlatIndex}]
As Remark \ref{RemReductionOnDirac},
we may assume that $D$ is a Dirac type operator.
Assume that $\mathrm{ind}_{G}^{}[D] = 0$
and we assume the converse.
that is,
for each $k \in \N$ there exits a Hilbert $B_{k}^{}$-module $G$-bundle $E_{}^{k}$ over $X$
whose curvature norm is less than $\frac{1}{k}$
satisfying that
$$
\mathrm{ind}_{G}^{}([E_{}^{k}]\hattensor [D]) \neq 0 
\quad \in K_{0}(C^{*}(G) \hattensor B_{k}^{}).
$$

To begin with, we have an exact sequence;
$$
0 \rightarrow
\dirsum{ B_{k}^{}}
\xrightarrow{\;\;\; \iota\;\;\;}
\prod B_{k}^{}
\xrightarrow{\;\;\; \pi \;\;\;}
\calq B_{k}^{}
\rightarrow 0,
$$
where $\iota$ and $\pi$ are natural inclusion and projection.
We also have the following exact sequence
\cite[Theorem T.6.26]{We93};
$$
0 \rightarrow
C_{\Max}^{*}(G)\hattensor_{\Max}^{} \parens{\dirsum{ B_{k}^{}}}
\xrightarrow{\;\;\; 1 \hattensor \iota\;\;\;}
C_{\Max}^{*}(G)\hattensor_{\Max}^{} \parens{\prod B_{k}^{}}
\xrightarrow{\;\;\;1 \hattensor \pi \;\;\;}
C_{\Max}^{*}(G)\hattensor_{\Max}^{} \parens{\calq B_{k}^{}}
\rightarrow 0.
$$
We have the exact sequence of $K$-groups
$$
K_{0}\parens{ C_{\Max}^{*}(G)\hattensor_{\Max}^{} \parens{\dirsum{ B_{k}^{}}} }
\to
K_{0}\parens{ C_{\Max}^{*}(G)\hattensor_{\Max}^{} \parens{\prod B_{k}^{}} }
\to
K_{0}\parens{ C_{\Max}^{*}(G)\hattensor_{\Max}^{} \parens{\calq B_{k}^{}} }.
$$
If all of $B_{k}^{}$ are commutative, then
$\calq B_{k}^{}$ is also commutative and hence nuclear. 
In that case, we also have the same exact sequences
in which $ C_{\Max}^{*}(G)$ and $\hattensor_{\Max}^{}$ are replaced by
$ C_{\red}^{*}(G)$ and $\hattensor_{\min}^{}$ respectively.

Let us start with $\ZZZZ
\in K_{0}\parens{C^{*}(G)\hattensor \parens{\prod B_{k}^{}}}$.
\\
Due to the flatness of $W$ (Proposition \ref{FlatnessOfW}) and
Proposition \ref{MuW0} and \ref{PiMu}, we have
$$(1 \hattensor \pi)_{*}\ZZZZ = \WWWW = 0.$$
It follows from the exactness that
there exists $\zeta \in
K_{0}\parens{C^{*}(G)\hattensor \parens{\dirsum{ B_{k}^{}}}}$
such that $$(1 \hattensor \iota)_{*}(\zeta) = \ZZZZ.$$
\begin{lem}
$A \hattensor \parens{ \bigoplus_{k\in \N} B_{k}^{} }$
is naturally isomorphic to
$\bigoplus_{k \in \N} \parens{ A \hattensor B_{k}^{} }
= {\displaystyle{\varinjlim_{n}
}} \bigoplus_{k=1}^{n} \parens{ A \hattensor B_{k}^{} }$.
\end{lem}
\begin{proof}
Let $C$ denote the direct product ${\displaystyle{\varinjlim_{n}
}} \bigoplus_{k=1}^{n} \parens{ A \hattensor B_{k}^{} }$.
Note that for the finite direct product,
we have the natural isomorphism 
$\bigoplus_{k=1}^{n} \parens{ A \hattensor B_{k}^{} } \cong
A \hattensor \parens{ \bigoplus_{k=1}^{n}  B_{k}^{} } $.
For each $n \in \N$, we have the following commutative diagram:
\newcommand{\AXtB}[1]
{ A \hattensor \parens{ \bigoplus_{k=1}^{#1} B_{k}^{} } }
\begin{align*}
\xymatrix{
\AXtB{n}
 \ar[r]^{\mathrm{id}_{A}^{} \hattensor \iota_{n}^{n+1}}
 \ar[rd]_{\mathrm{id}_{A}^{} \hattensor \iota_{n}^{}}
& \AXtB{n+1}
 \ar[d]^{\mathrm{id}_{A}^{} \hattensor \iota_{n+1}^{}}
\\
 & A \hattensor \parens{ \bigoplus_{k\in \N} B_{k}^{}}
}.
\end{align*}
Now by using the universal property of the direct limit,
we obtain a map $\phi$:
\begin{align*}
\xymatrix{
\AXtB{n} \ar[r] \ar[rd] 
&{\displaystyle{\varinjlim_{n}
}} A \hattensor \parens{ \bigoplus_{k=1}^{n} B_{k}^{} }
 \ar[d]^{\phi}
\\
 & A \hattensor \parens{ \bigoplus_{k\in \N} B_{k}^{}}
}.
\end{align*}
Since $\mathrm{id}_{A}^{} \hattensor \iota_{n}^{}$ are isometric and injective,
$\phi$ is isometric and injective on each sub-space $\AXtB{n} \subset {\displaystyle{\varinjlim_{n}
}} A \hattensor \parens{ \bigoplus_{k=1}^{n} B_{k}^{} }$.
Since the union of such sub-spaces is dense in ${\displaystyle{\varinjlim_{n}
}} A \hattensor \parens{ \bigoplus_{k=1}^{n} B_{k}^{} }$,
it follows that $\phi$ itself is isometric and injective.

As for the surjectivity of $\phi$, take any
 $a \hattensor \{ b_{k}^{} \} \in A \hattensor\parens{ \bigoplus_{k\in \N} B_{k}^{}}$.
For any $\ep>0$, there exists $n \in \N$ such that $\norm{b_{k}^{} }< \frac{\ep}{1+\norm{a}}$
for $k \geq n$.
Then replace $b_{k}^{}$ by $0$ for all $k \geq n$ to obtain an element
$\beta := \{ b_{1}^{}, b_{2}^{}, \ldots 0_{n}^{},  0_{n+1}^{},\ldots \} \in \bigoplus_{k\in \N} B_{k}^{}$.
Now we have that 
\begin{align*}
a \hattensor \beta 
=
(\mathrm{id}_{A}^{} \hattensor \iota_{n}^{}) ( a\hattensor \{ b_{1}^{}, b_{2}^{}, \ldots b_{n-1}^{} \})
=
\phi ( a\hattensor \{ b_{1}^{}, b_{2}^{}, \ldots b_{n-1}^{} \}) 
\in \Im(\phi)
\end{align*}
and $\norm{a\hattensor \{ b_{k}^{} \} - a \hattensor \beta} \leq
\norm{a} \norm{\{ b_{k}^{} \} - \beta} \leq \ep$. These imply that
$\Im(\phi)$ is dense in $A \hattensor \parens{ \bigoplus_{k\in \N} B_{k}^{} }$
and hence, $\phi$ is surjective since it has closed range.
\end{proof}
By this lemma, 
$C^{*}(G) \hattensor \parens{ \bigoplus B_{k}^{} }$
is naturally isomorphic to
$\bigoplus \parens{ C^{*}(G) \hattensor B_{k}^{} }$.
Besides, we have the natural isomorphism
$K_{0}\parens{\bigoplus \parens{ C^{*}(G) \hattensor B_{k}^{} } } \simeq
\bigoplus K_{0}\parens{ C^{*}(G) \hattensor B_{k}^{} }$
\cite[4.1.15 Proposition, 4.2.3 Remark]{HiRo00},
with the last $\bigoplus$ meaning the algebraic direct sum.
Thus we can consider the following diagram;
$$
\begin{CD}
K_{0}\parens{C^{*}(G) \hattensor \parens{ \bigoplus B_{k}^{} }}
@>{\iota_{*}}>>
K_{0}\parens{C^{*}(G)\otimes \parens{\prod B_{k}^{}}}
@>{(1 \otimes \pi)_{*}}>>
K_{0}\parens{C^{*}(G)\otimes \calq B_{k}^{}}
\\
@V{\{(1 \otimes p_k)_{*} \} }V{\cong}V
@V{\{(1 \otimes p_k)_{*} \} }VV
@.
\\
\bigoplus K_{0} \parens{ C^{*}(G) \hattensor B_{k}^{} }
@>>{\text{inclusion}}>
\prod K_{0}\parens{ C^{*}(G) \hattensor B_{k}^{} }
\end{CD}
$$
Since $p_k = \iota p_k$, this diagram commutes.
Note that both $\bigoplus$ and $\prod$ in the bottom row are
in the algebraic sense.
Again due to Proposition \ref{PiMu},
\begin{eqnarray*}
\braces{\XXXX{k}}_{k \in \N}
&=&
\{(1 \otimes p_k)_{*} \}\parens{\ZZZZ}
\\
&=&
\{(1 \otimes p_k)_{*} \}\parens{(1 \otimes \iota)_{*}(\zeta)}
\\
&=&
\{(1 \otimes p_k)_{*} \}(\zeta)
\quad \in \bigoplus K_{0} \parens{ C^{*}(G) \hattensor B_{k}^{} }.
\end{eqnarray*}
This implies that all of
$\XXXX{n} \in K_{0} \parens{ C^{*}(G) \hattensor B_{n}^{} }$ are equal to zero
except for finitely many $n \in \N$, which contradicts to our assumption.
\end{proof}
\subsection{On proof of Corollary \ref{CorAlmostFlatSgn}}
To prove Corollary \ref{CorAlmostFlatSgn}, we will combine 
Theorem \ref{ThmAlmostFlatIndex} with Theorem \ref{MainThm}.
Consider the same conditions as Theorem \ref{MainThm} on $X$, $Y$ and $G$
and assume additionally that $X$ and $Y$ are simply connected.
Let $\map{f} {Y} {X}$ be a $G$-equivariant orientation preserving homotopy invariant map.
Assume that
for each $k \in \N$ there exits a Hilbert $B_{k}^{}$-module $G$-bundle $E_{}^{k}$ over $X$
whose curvature norm is less than $\frac{1}{k}$
satisfying that
$$
\mathrm{ind}_{G}^{}([E_{}^{k}]\hattensor [\partial_{X}^{}]) \neq 
\mathrm{ind}_{G}^{}([f_{}^{*}E_{}^{k}]\hattensor [\partial_{Y}^{}])  
\quad \in K_{0}(C^{*}(G) \hattensor B_{k}^{}).
$$
as in the proof of Theorem \ref{ThmAlmostFlatIndex}.
Consider a $G$-manifold $Z := X \sqcup (-Y)$, 
the disjoint union of $X$ and orientation reversed $Y$
and the signature operator $\partial_{Z}^{}$ on it.
Although $Z$ is not connected, however,
we may apply Theorem \ref{ThmAlmostFlatIndex} to $\partial_{Z}^{}$,
after replacing some argument in the proof as follows.
Consider a family of Hilbert $B_{k}^{}$-module bundles
$\{ E_{}^{k} \sqcup f_{}^{*}E_{}^{k} \}$ over $Z$
and obtain a flat bundle $W \sqcup f_{}^{*}W$ as in subsection \ref{SectionAssembling}.
In order to obtain a global section $\overline{w}$ as in
(\ref{EqGlobalSectionOnW}) in the proof of Lemma \ref{LemIndTwistedByW},
we have used the connectedness of the base space.
In this case, construct a section $\overline{w} \colon X \to W$ on $X$ in the same way
 and pull back it on $Y$ by $f$ to obtain a global section on $Z$.
The other parts are the same as above.
\section*{Acknowledgements}
The author is supported by Natural Science Foundation of China (NSFC) 
Grant Number 11771143.

\end{document}